\documentclass[10pt,reqno,a4paper]{amsart}
\usepackage{amsmath,amsthm,verbatim,amscd,amssymb,setspace,enumitem,hyperref}
\usepackage{exscale,color}
\usepackage[colorinlistoftodos,prependcaption]{todonotes}
\usepackage{tikz-cd}
\usepackage{enumitem}
\usepackage{mathrsfs}

\renewcommand{\epsilon}{\varepsilon}
\newcommand{\N}{\mathbb{N}}

\newcommand{\R}{\mathbb{R}}
\newcommand{\C}{\mathbb{C}}

\newcounter{mtheorem}
\newtheorem{mtheorem}[mtheorem]{Theorem}

\newtheorem{mcor}[mtheorem]{Corollary}

\newcommand{\supp}{\operatorname{supp}}

\newcommand{{\vol}}{\rm vol}
\newcommand{\p}{\partial}

\newcommand{\Ric}{\operatorname{Ric}}

\newcommand{\Rm}{\operatorname{Rm}}

\def\tr{\operatorname{tr}}

\def\Rm{\operatorname{Rm}}

\newtheoremstyle{fancy}{}{}{\itshape}{}{\textbf\bgroup}{.\egroup}{ }{}
\newtheoremstyle{fancy2}{}{}{\rm}{}{\textbf\bgroup}{.\egroup}{ }{}

\theoremstyle{fancy}
\newtheorem{theorem}{Theorem}[section]
\newtheorem{lemma}[theorem]{Lemma}
\newtheorem{corollary}[theorem]{Corollary}

\newtheorem{prop}[theorem]{Proposition}

\theoremstyle{fancy2}
\newtheorem{definition}[theorem]{Definition}

\newtheorem{remark}[theorem]{Remark}

\newtheorem{claim}[theorem]{Claim}

\setlist{leftmargin=*}

\numberwithin{equation}{section}

\begin{document}
\title{Asymptotic Profiles and Non-Trivial Breathers in Kähler-Ricci Flow}
\author{Longteng Chen}
\address{Université Paris-Saclay, CNRS, Laboratoire de Mathématiques d'Orsay, 91405 Orsay, France }
\email{longteng.chen@universite-paris-saclay.fr}

\begin{abstract}
In this paper, we investigate the relationship between the long-time behavior of solutions to the Kähler–Ricci flow on an asymptotically conical gradient Kähler–Ricci expander $(M,g,X)$ and the asymptotic behavior of their initial data at spatial infinity.

After reviewing several fundamental concepts and basic results concerning the Kähler–Ricci flow, we explore the possibility that a global solution need not be asymptotically self-similar. More precisely, we demonstrate that the asymptotic profile of a solution may depend on the choice of time sequence tending to infinity. In particular, there exist solutions that are asymptotic to infinitely many distinct self-similar solutions along different time sequences.

We further establish an explicit connection between this phenomenon and the asymptotic behavior of the initial data under a corresponding family of dilations. Moreover, we show that a single solution can exhibit genuinely different asymptotic behaviors along different sequences of times tending to infinity, as well as under different choices of time-dependent rescalings.

Finally, we discuss the possibility of equipping the space of Kähler–Ricci flows on $(M,g,X)$ with a suitable topology such that the parabolic rescaling map of the Kähler–Ricci flows defines a dynamical system. We prove there are infinitely many non-soliton fixed points in this dynamical system, which are non-trivial Ricci breathers.
\end{abstract}
\maketitle
\section{Introduction}
This paper investigates the long-time behavior of the Kähler–Ricci flow on noncompact manifolds. Let $M$ be a complex manifold. A smooth family of Kähler metrics $(g(t))_{t \in (0,T)}$ on $M$ is said to solve the Kähler–Ricci flow equation if it satisfies the evolution equation
\begin{equation*}
    \frac{\partial}{\partial t} g(t) = -\Ric(g(t)), \quad \text{for all } t \in (0,T).
\end{equation*}
As a weakly parabolic equation, the Kähler–Ricci flow is expected to exhibit long-time behavior analogous to that of the parabolic equation. It is well-known that for a large number of linear and nonlinear equations the set of solutions is invariant under some group of space-time dilations. This phenomenon gives rise to the study of self-similar solutions, i.e. those solutions which are themselves invariant under the same group of dilations. A self-similar solution to the Kähler–Ricci flow is given by a Kähler–Ricci soliton. An \emph{expanding} (gradient) Kähler–Ricci soliton consists of a triple $(M,g,X)$, where $(M,g)$ is a Kähler manifold and $X$ is a real holomorphic vector field (of gradient type, $X=\nabla^g f$) satisfying the soliton equation:
\begin{equation*}
    \frac{1}{2}\mathcal{L}_X g = \Ric(g) + g.
\end{equation*}
In the gradient case, this equation can be rewritten as
\begin{equation*}
    \operatorname{Hess}_g f = \Ric(g) + g.
\end{equation*}
Let $\omega$ denote the Kähler form associated to $g$. Then the soliton equation is equivalently expressed at the level of $(1,1)$-forms as
\begin{equation*}
    i\partial\bar{\partial} f = \Ric(\omega) + \omega.
\end{equation*}
If $\Phi_t$ denotes the flow of $-\frac{X}{2t}$ for any $t>0$, then the \emph{self-similar} solution $g(t)_{t>0}$ can be defined as follows:
\begin{equation*}
    g(t):=t\Phi_t^*g.
\end{equation*}
This solution is self-similar because it is invariant under the following dilation group action.
\begin{definition}[Dilation group for K\"ahler-Ricci flow]\label{D}
  Let $h(t)_{t\in (0,T)}$ be a K\"ahler-Ricci flow on a K\"ahler-Ricci expander $(M,g,X)$, let $\Phi_s$ be the flow of $-\frac{X}{2s}$ for positive $s$. For every $\lambda>0$, the \emph{dilated K\"ahler-Ricci flow at the scale $\lambda$} is defined as follows:
    \begin{equation*}
        \mathcal{D}(\lambda)h(t):=\lambda\Phi_\lambda^*\left(h(\frac{t}{\lambda})\right),\quad \textnormal{for all $t\in(0,\lambda T)$}.
    \end{equation*}
    Moreover, the family $\{\mathcal D(\lambda)\}_{\lambda>0}$ forms a multiplicative group.
\end{definition}
The long-time behavior of the Kähler–Ricci flow on closed manifolds has been extensively studied by Cao in \cite{Cao1, Cao2}. In the noncompact setting, the stability of the Kähler–Ricci flow has been investigated in \cite{ChauKE, ChauSchnurer, longteng2}. In this paper, we focus on noncompact manifolds that are asymptotic to a cone at geometric infinity.
\begin{definition}[Asymptotically conical gradient K\"ahler-Ricci expander]\label{ACKR expander}
    An asymptotically conical gradient K\"ahler-Ricci expander is a triple $(M,g,X)$ being a complete expanding gradient K\"ahler-Ricci soliton
whose curvature $\operatorname{Rm}(g)$ satisfies
\begin{equation*}
\sup_{x\in M}|(\nabla^{g})^{k}\operatorname{Rm}(g)|_{g}(x)d_{g}(p,\,x)^{2+k}<\infty\quad\textrm{for all $k\in\mathbb{N}_{0}$,}
\end{equation*}
where $d_{g}(p,\,\cdot)$ denotes the distance to a fixed point $p\in M$ with respect to $g$. 
\end{definition}
On a noncompact complex manifold, for any complete Kähler metric $g_0$ with bounded curvature, Shi \cite{ShiJDG1} proved the existence and uniqueness of a complete solution ($\tilde g(t))_{t \in [0,T)}$ to the Kähler–Ricci flow with bounded curvature on every compact time interval. Our first theorem shows that, if the initial metric is chosen \emph{close} to a soliton metric, then Shi’s flow remains uniformly \emph{comparable} to the corresponding self-similar solution.
\begin{mtheorem}[Estimation along K\"ahler-Ricci flow]\label{main estimate thm}
   Let $(M,g,X)$ be an asymptotically conical gradient K\"ahler-Ricci expander as in Definition \ref{ACKR expander}. There exists an $a>0$, such that for any $g_0$ a K\"ahler metric that is $a-$close to $g$ (see Definition \ref{a closeness}), the following results hold for the unique Shi's K\"ahler-Ricci flow $\tilde g(t)_{t\in[0,T)}$ with $\tilde g(0)=g_0$:
    \begin{enumerate}[label=\textnormal{(\alph{*})}, ref=(\alph{*})]
        \item \label{a condition}The K\"ahler-Ricci flow $\tilde g(t)_{t\in[0,T)}$ is immortal, that is, $T=\infty$.
        \item\label{b condition} The flow $\tilde g(\cdot)$ is comparable to the self-similar flow $g(\cdot+1)$, i.e.
        \begin{enumerate}[label=\textnormal{(\roman{*})}, ref=(\roman{*})]
             \item There exists a uniform constant $C>1$ such that along the flow, we have
        \begin{equation*}
           C^{-1}g(t+1)\le \tilde g(t)\le Cg(t+1).
        \end{equation*}
        \item Along the flow, let $p\in M$ be a zero point of $X$, we have
        \begin{equation*}
            \sup_{t\ge0}\sup_M\left(\sqrt{t+1}+d_{g(t+1)}(p,\cdot)\right)^k|(\nabla^{g(t+1)})^k\tilde g(t)|_{g(t+1)}\le C_k,\quad\textnormal{for all $k\in\N_0$},
        \end{equation*}
        where $d_{g(t+1)}(p,\,\cdot)$ denotes the distance to a fixed point $p\in M$ with respect to $g(t+1)$. 
        \end{enumerate}
    \end{enumerate}
\end{mtheorem}
Due to the parabolic rescaling properties of general parabolic equations, it is widely believed that the long-time behavior of solutions is governed by the asymptotic behavior of the initial data at spatial infinity (see \cite{CazenaveDicksteinWeisslerNS} for the Navier–Stokes system and \cite{CazenaveDicksteinWeisslerParaEq} for the heat equation). On an AC (asymptotically conical) Kähler–Ricci expander, the asymptotic behavior of the Kähler metric is governed by the dilated K\"ahler metric at small scales, induced by the natural dilation structure of the asymptotic cone.
\begin{definition}[Dilation group for K\"ahler metrics]\label{D'}
    Let $h$ be a K\"ahler metric on $(M,g,X)$, let $\Phi_s$ be the flow of $-\frac{X}{2s}$ for positive $s$. For every $\lambda>0$, the \emph{dilated K\"ahler metric at scale $\lambda$} is defined as follows:
    \begin{equation*}
        \mathcal{D}'(\lambda)h:=\lambda\Phi_\lambda^*h.
    \end{equation*}
     Moreover, the family $\{\mathcal D(\lambda)\}_{\lambda>0}$ forms a multiplicative group.
\end{definition}
Letting the small-scale parameters in Definitions \ref{D} and \ref{D'} tend to zero yields the corresponding asymptotic profiles for both the Kähler–Ricci flow and the initial Kähler metric.
\begin{definition}[$\varkappa-$limit set and $\kappa-$limit set]
   Let $h_0$ be a K\"ahler metric on $(M,g,X)$ with bounded curvature, assume that $h(t)_{t\in[0,T)}$ is the unique Shi's K\"ahler-Ricci flow on $(M,g,X)$ starting from $h_0$. Furthermore, we assume that this flow is immortal, i.e. $T=\infty$. The $\varkappa-$limit set and $\kappa-$limit set are defined as follows:
    \begin{equation*}
       \begin{split}
          & \varkappa(h_0):=\{\tilde h(t)_{t\in (0,\infty)}\ |\ \exists \lambda_n\to 0^+, \textnormal{ such that $\mathcal{D}(\lambda_n)h(\cdot)$ converge to $\tilde h(\cdot)$ locally smoothly on $M\times (0,\infty)$}\},\\
          &\kappa(h_0):=\{\tilde h_0\ |\ \exists \lambda_n\to 0^+, \textnormal{ such that $\mathcal{D}'(\lambda_n)h_0$ converge to $\tilde h_0$ locally smoothly on $M\setminus E $}\}.
       \end{split}
    \end{equation*}
    Here $E\subset M$ is the exceptional divisor as in Theorem \ref{theorem of conlonderuellesun}.
\end{definition}
An asymptotic cone of a metric space $(M,d)$ is any pointed Gromov–Hausdorff limit of a sequence of rescaled metric spaces $(M,\lambda_i d,p)$, where $p\in M$ and $\lambda_i\to 0^+$. We refer the reader to \cite{BuragoBuragoSergei} for a precise definition of the pointed Gromov–Hausdorff topology. With strong geometric assumptions, one can show the uniqueness of asymptotic cone as \cite{TianCheeger,ColdingMinicozzi1,ChowLu}. For the non uniqueness part and the classification results, see \cite{Kota,ZhouSX2,ZhouSX1}. From the above definition, under suitable hypotheses one can show $\kappa(h_0)$ is a subset of the asymptotic cones of $(M,d_{h_0})$. A priori, $\kappa(h_0)$ may contain infinitely many elements.

The following theorem shows that the long-time behavior of the Kähler–Ricci flow is determined by the asymptotic spatial behavior of the initial metric.
\begin{mtheorem}[Relationship between $\varkappa-$limit set and $\kappa-$limit set]\label{relation between two limit sets}
    Let $g_0$ be the K\"ahler form as in Theorem \ref{main estimate thm}, then its $\varkappa-$limit set and $\kappa-$limit set are nonempty. Moreover, for each $\tilde h_0\in \kappa(\omega_0)$ (resp. $\tilde h(t)_{t\in (0,\infty)}\in \varkappa(g_0)$), there exists a K\"ahler-Ricci flow $\tilde h(t)_{t\in (0,\infty)}\in \varkappa(g_0)$ (resp. a K\"ahler metric $\tilde h_0\in \kappa(g_0)$) such that:
    \begin{enumerate}[label=\textnormal{(\alph{*})}, ref=(\alph{*})]
    \item\label{Theorem B i} The flow $\tilde h(t)_{t\in (0,\infty)}$ is with type III singularity and comparable to the self-similar flow $g(t)_{t\in (0,\infty)}$:
    \begin{enumerate}[label=\textnormal{(\roman{*})}, ref=(\roman{*})]
        \item There exists a uniform constant $C>1$ such that along the flow, we have
        \begin{equation*}
           C^{-1}g(t)\le \tilde h(t)\le Cg(t),\quad \forall t>0.
        \end{equation*}
        \item Along the flow, let $p\in M$ be a zero point of $X$, we have
        \begin{equation*}
            \sup_{t>0}\sup_M\left(\sqrt{t}+d_{g(t)}(p,\cdot)\right)^k|(\nabla^{g(t)})^k\tilde h(t)|_{g(t)}\le C_k,\quad \textnormal{for all $k\in\N_0$}
        \end{equation*}
        where $d_{g(t)}(p,\,\cdot)$ denotes the distance to a fixed point $p$ with respect to $g(t)$. 
    \end{enumerate}
    \item \label{Theorem B ii} The metrics $\tilde h(t)$ converge locally smoothly to $\tilde h_0$ on $M \setminus E$.
    
    \item \label{Theorem B iii} Let $\pi: M\mapsto\mathcal{C}$ be the K\"ahler resolution as in Theorem \ref{theorem of conlonderuellesun}. Then, 
    the metric completion of $(\mathcal{C} \setminus \{o\}, \pi_* \tilde h_0)$ is homeomorphic to $\mathcal{C}$; we denote it by $(\mathcal{C}, d_{\pi_* \tilde h_0})$. Moreover, for each $p \in E$, $o$ being the vertex of $\mathcal{C}$,
    \begin{equation*}
        \left(M, d_{\tilde h(t)}, p\right) \xrightarrow{t \to 0^+} \left(\mathcal{C}, d_{\pi_* \tilde h_0}, o\right),
    \end{equation*}
    in the pointed Gromov--Hausdorff topology.
    \item \label{Theorem B iv} In particular, if $(\mathcal{C},\pi_*\tilde h_0)$ is a K\"ahler cone, then $(M,\tilde h(1),X)$ is the unique gradient AC K\"ahler-Ricci expander whose self-similar solution converges to $\pi_*\tilde h_0$ locally smoothly on $\mathcal{C}\setminus \{o\}$.
\end{enumerate}
\end{mtheorem}
  Due to the results in \cite{ConlonDeruelleJDG}, if $(\mathcal{C},\pi_*\tilde h_0)$ is a K\"ahler cone (see Subsection \ref{subsection of K cones} for definitions), there exists $(M,\tilde h,X)$ an AC gradient K\"ahler-Ricci expander with bounded curvature such that the corresponding self-similar solution converges locally smoothly to $\pi_*\tilde h_0$ on $\mathcal{C}\setminus\{o\}$. By the work of \cite{ConlonDeruelleSun,longteng1}, such a gradient K\"ahler-Ricci expander $(M,\tilde h,X)$ is in fact unique. Therefore, \ref{Theorem B iv} in Theorem \ref{relation between two limit sets} implies that $(M,\tilde h,X)=(M,\tilde h(1),X)$.
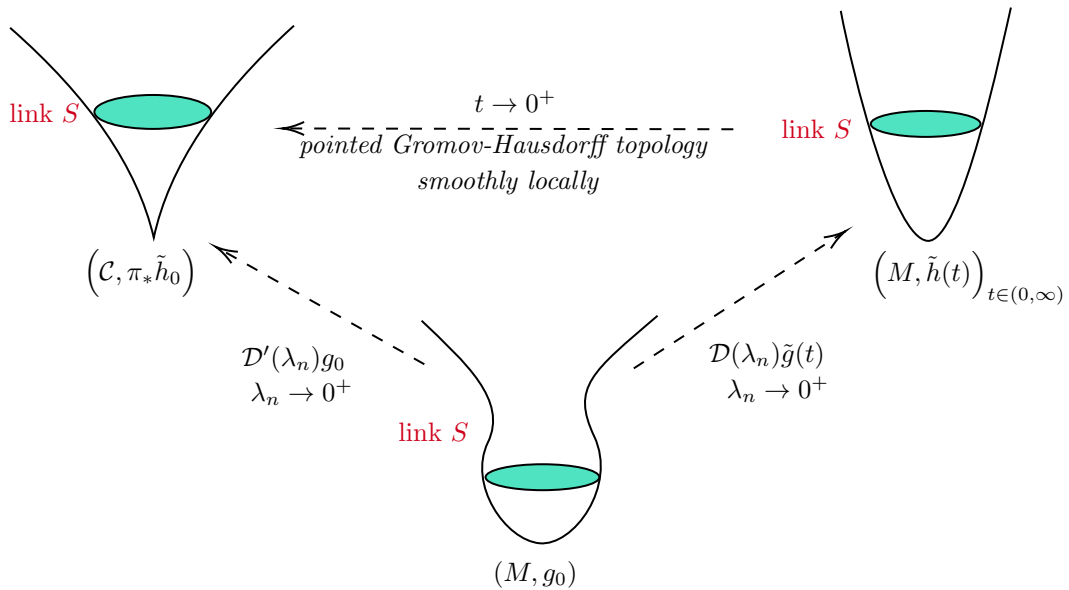
\begin{center}

\tikzset{every picture/.style={line width=0.75pt}} 

\begin{tikzpicture}[x=0.75pt,y=0.75pt,yscale=-1,xscale=1]

\draw  [color={rgb, 255:red, 0; green, 0; blue, 0 }  ,draw opacity=1 ][line width=0.75]  (254.5,83.28) .. controls (215.28,118.81) and (191.87,154.21) .. (184.25,189.5) .. controls (176.06,154.33) and (152.06,119.28) .. (112.26,84.34) ;
\draw    (528.47,75.86) .. controls (534.73,109.07) and (554.84,189.48) .. (571.9,191.08) .. controls (588.95,192.68) and (609.73,91.15) .. (613.72,74.2) ;
\draw    (318.45,231.09) .. controls (340.16,251.82) and (361.88,270.88) .. (351.66,292.08) .. controls (341.44,313.28) and (360.5,340.99) .. (377.56,342.59) .. controls (394.61,344.19) and (416.64,313.9) .. (404.94,289.08) .. controls (393.24,264.25) and (403.7,256.79) .. (436.67,228.6) ;
\draw  [dash pattern={on 4.5pt off 4.5pt}]  (319.82,252.9) -- (217.01,195.02) ;
\draw [shift={(215.26,194.04)}, rotate = 29.38] [color={rgb, 255:red, 0; green, 0; blue, 0 }  ][line width=0.75]    (10.93,-3.29) .. controls (6.95,-1.4) and (3.31,-0.3) .. (0,0) .. controls (3.31,0.3) and (6.95,1.4) .. (10.93,3.29)   ;
\draw  [dash pattern={on 4.5pt off 4.5pt}]  (424.77,254.97) -- (529.25,187.25) ;
\draw [shift={(530.93,186.16)}, rotate = 147.05] [color={rgb, 255:red, 0; green, 0; blue, 0 }  ][line width=0.75]    (10.93,-3.29) .. controls (6.95,-1.4) and (3.31,-0.3) .. (0,0) .. controls (3.31,0.3) and (6.95,1.4) .. (10.93,3.29)   ;
\draw  [dash pattern={on 4.5pt off 4.5pt}]  (474.63,135.18) -- (250.64,134.77) ;
\draw [shift={(248.64,134.77)}, rotate = 0.11] [color={rgb, 255:red, 0; green, 0; blue, 0 }  ][line width=0.75]    (10.93,-3.29) .. controls (6.95,-1.4) and (3.31,-0.3) .. (0,0) .. controls (3.31,0.3) and (6.95,1.4) .. (10.93,3.29)   ;
\draw  [color={rgb, 255:red, 0; green, 0; blue, 0 }  ,draw opacity=1 ][fill={rgb, 255:red, 80; green, 227; blue, 194 }  ,fill opacity=1 ] (154.5,126.5) .. controls (154.5,131.19) and (167.6,135) .. (183.75,135) .. controls (199.9,135) and (213,131.19) .. (213,126.5) .. controls (213,121.81) and (199.9,118) .. (183.75,118) .. controls (167.6,118) and (154.5,121.81) .. (154.5,126.5) -- cycle ;
\draw  [color={rgb, 255:red, 0; green, 0; blue, 0 }  ,draw opacity=1 ][fill={rgb, 255:red, 80; green, 227; blue, 194 }  ,fill opacity=1 ] (350.5,309.5) .. controls (350.5,313.09) and (363.26,316) .. (379,316) .. controls (394.74,316) and (407.5,313.09) .. (407.5,309.5) .. controls (407.5,305.91) and (394.74,303) .. (379,303) .. controls (363.26,303) and (350.5,305.91) .. (350.5,309.5) -- cycle ;
\draw  [color={rgb, 255:red, 0; green, 0; blue, 0 }  ,draw opacity=1 ][fill={rgb, 255:red, 80; green, 227; blue, 194 }  ,fill opacity=1 ][line width=0.75]  (543.5,132.5) .. controls (543.5,136.09) and (555.81,139) .. (571,139) .. controls (586.19,139) and (598.5,136.09) .. (598.5,132.5) .. controls (598.5,128.91) and (586.19,126) .. (571,126) .. controls (555.81,126) and (543.5,128.91) .. (543.5,132.5) -- cycle ;

\draw (147.65,193.76) node [anchor=north west][inner sep=0.75pt]    {$\left(\mathcal{C} ,\pi _{*}\tilde{h}_{0}\right)$};
\draw (110.51,120.79) node [anchor=north west][inner sep=0.75pt]  [color={rgb, 255:red, 208; green, 2; blue, 27 }  ,opacity=1 ] [align=left] {link $\displaystyle S$};
\draw (352.81,349.75) node [anchor=north west][inner sep=0.75pt]    {$( M,g_{0})$};
\draw (497.26,128.89) node [anchor=north west][inner sep=0.75pt]  [color={rgb, 255:red, 208; green, 2; blue, 27 }  ,opacity=1 ] [align=left] {link $\displaystyle S$};
\draw (542.13,194.6) node [anchor=north west][inner sep=0.75pt]    {$\left( M,\tilde{h}( t)\right)_{t\in ( 0,\infty )}$};
\draw (305.41,280.63) node [anchor=north west][inner sep=0.75pt]  [color={rgb, 255:red, 208; green, 2; blue, 27 }  ,opacity=1 ] [align=left] {link $\displaystyle S$};
\draw (220.17,239.97) node [anchor=north west][inner sep=0.75pt]    {$ \begin{array}{l}
\mathcal{D} '( \lambda _{n}) g _{0}\\
\ \lambda _{n}\rightarrow 0^{+}
\end{array}$};
\draw (224.37,253.99) node [anchor=north west][inner sep=0.75pt]    {$$};
\draw (454.33,238.15) node [anchor=north west][inner sep=0.75pt]    {$ \begin{array}{l}
\mathcal{D}( \lambda _{n})\tilde{g }( t)\\
\ \ \lambda _{n}\rightarrow 0^{+}
\end{array}$};
\draw (343.47,114.23) node [anchor=north west][inner sep=0.75pt]    {$t\rightarrow 0^{+}$};
\draw (255.4,136.08) node [anchor=north west][inner sep=0.75pt]   [align=left] {\emph{pointed Gromov-Hausdorff topology}\\};
\draw (314.47,154.45) node [anchor=north west][inner sep=0.75pt]   [align=left] {\emph{smoothly locally}};

\end{tikzpicture}
\end{center}
The above picture illustrates the following phenomenon: let $\{\lambda_n\}_{n\in\mathbb{N}_0}$ be a sequence of positive numbers with $\lambda_n \to 0^+$. Suppose that the dilated Kähler metrics $\mathcal{D}'(\lambda_n)g_0$ converge locally smoothly to $\tilde h_0$ on $M \setminus E$. In other words, $g_0$ is asymptotic to $\tilde h_0$ at geometric infinity along the sequence $\{\lambda_n\}_{n\in\N_0}$. Then, up to passing to a subsequence, the dilated Kähler–Ricci flows $\mathcal{D}(\lambda_n)\tilde g(t)$ converge locally smoothly to the Kähler–Ricci flow emanating from $(\mathcal{C}, \pi_* \tilde h_0)$, in the sense of both pointed Gromov–Hausdorff convergence and local smooth convergence. Conversely, the reverse implication also holds.

The K\"ahler-Ricci flow $\tilde g(t)$ in Theorem \ref{main estimate thm} is called \emph{asymptotically self-similar} if $\varkappa(g_0)$ contains only a single point, which must then represent a self-similar solution.
Likewise, the initial Kähler metric $g_0$ is said to be \emph{asymptotically conical} if $\kappa(\omega_0)$ reduces to a singleton.
\begin{mcor}\label{singleton}
    If $g_0$ is $a-$close to $g$ as in Theorem \ref{main estimate thm}, then the flow $\tilde g(t)_{t\in [0,\infty)}$  is asymptotically self-similar if and only if $g_0$ is asymptotically conical. When in this case, let $(\mathcal C,g_\mathcal{C})$ be a K\"ahler cone which is the singular model of $(M,g,X)$ as in Theorem \ref{theorem of conlonderuellesun},
    \begin{equation*}
        \varkappa(g_0)=\{(g'(t))_{t>0}\};\quad \kappa(g_0)=\{\pi^*g_\mathcal{C}'\}.
    \end{equation*}
    Here $g_\mathcal{C}'$ is a K\"ahler conical metric defined on $\mathcal{C}$, $(M,g'(1),X)$ is the unique gradient K\"ahler-Ricci expander with bounded curvature such that the corresponding self-similar solution $\pi_*g'(t)$ converges locally smoothly to $g_\mathcal{C}'$ on $\mathcal{C}\setminus\{o\}$ as $t$ tends to 0.
\end{mcor}
A byproduct of Theorem \ref{main estimate thm} and \ref{relation between two limit sets} is the existence of Kähler–Ricci flows with singular initial data that are not necessarily conical. Indeed, for a singular metric that remains \emph{close} to a Kähler cone metric $\omega_\mathcal{C}$ one can naturally construct a Kähler–Ricci flow emerging from it.
\begin{mcor}[K\"ahler-Ricci flow with singular data]\label{singular data}
 Let $(\mathcal C,g_\mathcal{C})$ be a K\"ahler cone which is the singular model of $(M,g,X)$ as in Theorem \ref{theorem of conlonderuellesun}. There exists an $a'>0$  such that for all $\omega_{\operatorname{sing}}=\omega_\mathcal{C}+i\partial\bar\partial\varphi_{\operatorname{sing}}$ being a smooth K\"ahler form on $\mathcal{C}\setminus\{o\}$ with the following properties:
    \begin{enumerate}
        \item Let $r$ denote the radial function, we have
        \begin{equation*}
           \sup_{\mathcal{C}\setminus\{o\}} \max\{r^{-2}|\varphi_{\operatorname{sing}}|,r^{-1}|\nabla^{g_\mathcal C}\varphi_{\operatorname{sing}}|_{g_\mathcal{C}},|(\nabla^{g_\mathcal{C}})^2\varphi_{\operatorname{sing}}|_{g_\mathcal{C}}\}\le a'.
        \end{equation*}
        \item For each $k\in \N^*$, there exists a constant $C_k>0$ such that
        \begin{equation*}
           \sup_{\mathcal{C}\setminus\{o\}}  r^k|(\nabla^{g_\mathcal{C}})^{k+2}\varphi_{\operatorname{sing}}|_{g_\mathcal{C}}\le C_k.
        \end{equation*}
        \item The K\"ahler form $\omega_{\operatorname{sing}}$ is $Jr\partial_r-$invariant.
    \end{enumerate}
    Then there exists an immortal smooth K\"ahler-Ricci flow $\omega_{\operatorname{smooth}}(t)_{t\in (0,\infty)}$ on $M$ such that \ref{Theorem B i}-\ref{Theorem B iv} in Theorem \ref{relation between two limit sets} hold for $(g_{\operatorname{smooth}}(t))_{t\in (0,\infty)}$ on $M$ and $\pi^*\omega_{\operatorname{sing}}$ on $M\setminus E$.
\end{mcor}
In this article, we are motivated by the question of
whether or not asymptotically self-similar behavior is generic. In terms of dynamical systems, this amounts to studying the $\varkappa$-limit set of general solutions. Is the typical $\varkappa$-limit set a singleton or do most solutions have a bigger $\varkappa$-limit set? In fact, there exist solutions whose $\varkappa$-limit set, in terms of self-similar variables (see Definition \ref{self-similar variables}), are dense in a metric ball of AC expanders. This is the extreme opposite of a solution whose $\varkappa$-limit is a singleton, i.e. an asymptotically self-similar solution in Corollary \ref{singleton}. Before proving the existence of such \emph{universal initial data}, we equip the space of AC gradient expanders with a natural metric structure.
\begin{definition}
 Let $\xi$ be the Reeb vector field on the K\"ahler cone $(\mathcal{C},g_{\mathcal{C}})$ being the singular model of $(M,g,X)$. For all $\delta>0$, one can define the following metric ball $(\mathcal{B}(g_\mathcal{C},\delta),\operatorname{BL})$ on the moduli space of K\"ahler conical metrics on $\mathcal C$:
    \begin{equation*}  
            \mathcal{B}(g_\mathcal{C},\delta)=\{g'_\mathcal{C}\ | \ g'_\mathcal{C} \textnormal{ is a $\xi-$invariant K\"ahler conical metric},\ \operatorname{BL}_{\mathcal{C}}(g_\mathcal C,g_\mathcal C')\le \delta\}.    
    \end{equation*} 
    Here, for two metrics $g_1,g_2$ defined on $\mathcal{C}\setminus\{o\}$,
    \begin{equation*}
        \operatorname{BL}_{\mathcal{C}}(g_1,g_2):=\sup_{x\in \mathcal{C}\setminus\{o\}}\sup\{\delta\in\R\ |\ e^{-\delta}g_2(x)\le g_1(x)\le e^{\delta}g_2(x)\}.
    \end{equation*}
\end{definition}
As in \cite{ConlonDeruelleJDG,ConlonDeruelleSun} and \cite{longteng1}, it is shown that for each $g_\mathcal{C}'\in  \mathcal{B}(g_\mathcal{C},\delta)$, there is the unique gradient K\"ahler-Ricci expander $(M,g',X)$ with bounded curvature such that the corresponding self-similar solution $\pi_*g'(t)$ converges locally smoothly to $g_\mathcal{C}'$ on $\mathcal{C}\setminus\{o\}$. We denote this correspondence by $\varrho$, and we endow $\varrho( \mathcal{B}(g_\mathcal{C},\delta))$ with the pseudometric $\operatorname{BL}_\infty$ (see Definition \ref{bi-lipschitz norm at infinity}).
\begin{lemma}\label{isometry lemma}
   The space $(\varrho( \mathcal{B}(g_\mathcal{C},\delta)),\operatorname{BL}_\infty)$ is a metric space for any $\delta>0$. Moreover, the map
   \begin{equation*}
     \varrho: (\mathcal{B}(g_\mathcal{C},\delta),\operatorname{BL})\mapsto(\varrho( \mathcal{B}(g_\mathcal{C},\delta)),\operatorname{BL}_\infty)
   \end{equation*}
   is an isometry.
\end{lemma}
\begin{mtheorem}[Universal initial data]\label{universal initial data}
    There exists a $\delta>0$, and $g_0$ a K\"ahler metric on $M$ 
    such that $\kappa(g_0)$ is dense in the metric ball $(\mathcal{B}(g_\mathcal{C},\delta),\operatorname{BL})$. Consequently, its $\varkappa-$limit set $\varkappa(g_0)$ is dense in the metric ball $(\varrho( \mathcal{B}(g_\mathcal{C},\delta)),\operatorname{BL}_\infty)$.
\end{mtheorem}
Finally, we study the K\"ahler--Ricci flows on $(M,g,X)$ from the perspective of dynamical systems. We introduce a metric space $(\mathcal{N}_C^B,d)$ consisting of immortaal K\"ahler--Ricci flows on $(M,g,X)$ such that, together with dilation group of the K\"ahler--Ricci flow $\mathcal{D}(\lambda)$ for $\lambda\ge 1$, it gives rise to a topological dynamical system. Moreover, this dynamical system is semi-conjugated to a dynamical system of K\"ahler metrics on the K\"ahler cone. We refer the reader to Section \ref{section dyn system} for a precise discussion. As a direct consequence, we obtain the existence of infinitely many non-trivial Ricci breathers on $(M,g,X),$ i.e. the Ricci flow which is not the self-similar solutions associated to K\"ahler-Ricci expanders:

\begin{mtheorem}[Non-trivial Ricci breathers]\label{breathers}
  For each $\lambda>1$, on the AC gradient K\"ahler-Ricci expander $(M,g,X)$ there are infinitely many non-trivial Ricci breathers, which are immortal non self-similar K\"ahler-Ricci flows with type III singularity, such that each $\tilde \omega(t)_{t>0}$ of them satisfies
   \begin{equation*}
       \lambda\Phi_\lambda^*\tilde \omega\left(\frac{t}{\lambda}\right)=\tilde\omega(t).
   \end{equation*}
Here $\Phi_t$ is the flow of $-\frac{X}{2t}$ for $t>0$.
\end{mtheorem}
Perelman proved in \cite{PerelmanEntropy} that, on compact manifolds, every Ricci breather must be a self-similar solution. In the noncompact setting, we refer the reader to \cite{ToppingNonBreather,QSZhangNonbreathers} for further results concerning the existence and non-existence of nontrivial Ricci breathers. In contrast to the breathers constructed by Topping in \cite{ToppingNonBreather}, the nontrivial breathers appearing in Theorem \ref{breathers} are K\"ahler and irreducible.
\subsection{Outline of proof}
We now briefly explain the main steps of the proof and outline the structure of the paper. In Section \ref{section preliminaire}, we review the relevant definitions and properties of K\"ahler cones and expanding gradient K\"ahler--Ricci solitons.

In Section \ref{estimation of KRF}, we prove Theorem \ref{main estimate thm}. The core of the analysis consists of establishing a priori estimates for several quantities along the K\"ahler--Ricci flow. The techniques developed in this section may be of independent interest; see also \cite{longteng2,ChenHallgrenLucas} for related computations.

In Section \ref{Long-time behaviors of flows and asymptotic behaviors of initial values}, we prove Theorem \ref{relation between two limit sets} together with Corollaries \ref{singleton} and \ref{singular data}, relying on the estimates established in Theorem \ref{estimation of KRF}.
In Section \ref{universal initial values}, we construct the universal initial value appearing in Theorem \ref{universal initial data}. 
\section{Preliminaries}\label{section preliminaire}
\subsection{K\"ahler cone}\label{subsection of K cones}
For us, the definition of a Riemannian cone will take the following form.
\begin{definition}\label{cone}
Let $(S, g_{S})$ be a closed Riemannian manifold. The \emph{Riemannian cone} $\mathcal{C}$ with \emph{link} $S$ is defined to be $\R_{>0}\times S$ with metric $g_\mathcal{C} = dr^2 \oplus r^2g_{S}$ up to isometry. The radius function $r$ is then characterised
intrinsically as the distance from the apex in the metric completion.
\end{definition}
The following is a simple computation.

\begin{lemma}\label{scalar}
Let $(S,\,g_{S})$ be a closed Riemannian manifold of real dimension $m$ and let $(\mathcal{C},g_\mathcal{C})$ be the Riemannian cone with link $S$ and radial function $r$. Then the Ricci curvature $\Ric(g_{S})$ of $g_{S}$ and the Ricci curvature $\Ric(g_{\mathcal{C}})$ of the cone metric $g_{\mathcal{C}}$ over $(S,\,g_{S})$ are related by
$$\Ric(g_{\mathcal{C}})=\Ric(g_{S})-(m-1)g_{S}.$$
In particular, the scalar curvatures $R_{g_{0}}$ and $R_{g_{S}}$ of $g_{0}$ and $g_{S}$ respectively are related by
\begin{equation*}
R_{g_{\mathcal{C}}}=\frac{1}{r^{2}}\left(R_{g_{S}}-m(m-1)\right).
\end{equation*}
\end{lemma}

We may further impose that a Riemannian cone is K\"ahler, as the next definition demonstrates.
\begin{definition}A \emph{K{\"a}hler cone} is a Riemannian cone $(\mathcal{C},g_\mathcal{C})$ such that $g_\mathcal{C}$ is K{\"a}hler, together with a choice of $g_\mathcal{C}$-parallel complex structure $J_\mathcal{C}$. This will in fact often be unique up to sign. We then have a K{\"a}hler form $\omega_\mathcal{C}(X,Y) = g_\mathcal{C}(J_\mathcal{C}X,Y)$, and $\omega_\mathcal{C} = \frac{i}{2}\p\bar{\p} r^2$ with respect to $J_\mathcal{C}$.
\end{definition}

The vector field $r\partial_{r}$ is real holomorphic and $\xi:=J_\mathcal{C}r\partial_r$ is real holomorphic and Killing \cite[Appendix A]{MartelliSparksYau}. This latter vector field is known as the \emph{Reeb
vector field}. The closure of its flow in the isometry
group of the link of the cone generates the holomorphic isometric action of a real torus on $C_{0}$ that
fixes the apex of the cone. Furthermore, every K\"ahler cone is affine algebraic \cite[Theorem 2.4]{ConlonDeruelleSun}.
\subsection{Asymptotically conical gradient K\"ahler-Ricci expander}
In this section we summarize some basic properties about $(M,g,X)$ which is an AC gradient K\"ahler-Ricci expander, for the details of proof, see \cite{longteng1,longteng2,ChenHallgrenLucas}.
\begin{theorem}[{\cite[Theorem A]{ConlonDeruelleSun}} Structure theorem]\label{theorem of conlonderuellesun}
Let $(M,\,g,\,X)$ be an AC complete gradient K\"ahler-Ricci expander, then there exists a K\"ahler resolution $\pi:M\to \mathcal{C}$ with exceptional set $E$ such that the following hold:
\begin{enumerate}
  \item The variety $\mathcal{C}$ is a K\"ahler cone endowed with K\"ahler conical metric $g_\mathcal{C}$.
  \item Let $r$ denote the radial function on $\mathcal{C}$, then $d\pi(X)=r\partial_{r}$ and
\begin{equation*}
|(\nabla^{g_{\mathcal{C}}})^k(\pi_{*}g-g_{\mathcal{C}})|_{g_{\mathcal{C}}} \leq C_{k}r^{-2-k}\quad\textrm{for all $k\in\mathbb{N}_{0}$}.
\end{equation*}
\end{enumerate}
\end{theorem}
From now on, we identify $M\setminus E$ with $\mathcal{C}\setminus\{o\}$ by the K\"ahler resolution $\pi$. The radial function $r$ can be extended by $0$ as a continuous function on $M$. The first results are the celebrated soliton identities (see \cite[Section 2 of Chapter 1]{ChowEtAl}).
\begin{lemma}[Soliton identities]\label{soliton indentities}
Let $f$ be a Hamiltonian of the vector field $X$, then we have
    \begin{equation*}
        \begin{split}
            &\Delta_\omega f=n+R_\omega,\\
            &\nabla^g R_\omega+\Ric(g)(X)=0,\\
            &|\partial f|_g^2+R_\omega=f+\textrm{constant}.
        \end{split}
    \end{equation*}
    Here $n=\dim_\C M$, $\Delta_\omega$ is the K\"ahler Laplacian, $R_\omega=\frac{1}{2}R_g$ is the K\"ahler scalar curvature, $R_g$ is the scalar curvature of $g$.

\end{lemma}
 We normalized $f$ such that $f=|\partial f|_g^2+R_\omega+n$, then we naturally have $(\Delta_\omega+\frac{X}{2})f=f$.
 \begin{prop}\label{lower bound of f}
    There exists an $\varepsilon>0$ such that 
    \begin{equation*}
        \inf_M (R_\omega+n)\ge\varepsilon>0.
    \end{equation*}
    Then it follows that on $M$,
    \begin{equation*}
        \Delta_\omega f=R_\omega+n\ge\varepsilon,\quad f=|\partial f|_g^2+R_\omega+n\ge\varepsilon.
    \end{equation*}
\end{prop}
\begin{proof}
    See \cite[Lemma 2.3]{longteng1}.
\end{proof}
 \begin{lemma}\label{proper function}
    The normalized potential $f$ is a proper function, and for a fixed point $p\in M$ such that $X(p)=0$, there exist constants $C,c_1,c_2>0$ such that $\forall x\in M,$
    \begin{equation}\label{eq: d and f}
       \begin{split}
            \frac{(d_g(p,x)-c_1)^2}{2}-C\le &f(x)\le  \frac{(d_g(p,x)+c_2)^2}{2},\\
            \frac{r(x)^2}{2}\le &f(x)\le \frac{r(x)^2}{2}+C.
       \end{split}
    \end{equation}
\end{lemma}
\begin{proof}
    See \cite[Proposition 2.20]{ConlonDeruelleSun} and \cite[Corollary 2.14]{ChenHallgrenLucas}. Moreover, from \eqref{eq: d and f}, together with Proposition \ref{lower bound of f}, we deduce the existence of uniform constant $A>0$ such that
    \begin{equation*}
        d_g(p,\cdot)\le A f,
    \end{equation*}
    holds on $M$.
\end{proof}
\begin{corollary}\label{proper function with t}
    Let $g(t)$ be the self-similar solution associated to $(M,g,X)$, $\Phi_t$ be the flow of $-\frac{X}{2t}$ for all $t>0$, and let $p\in M$ such that $X(p)=0$; then $\forall x\in M,$
    \begin{equation*}
        \begin{split}
            \frac{(d_{g(t)}(p,x)-c_1\sqrt{t})^2}{2}-Ct\le &tf(\Phi_t(x))\le\frac{(d_{g(t)}(p,x)+c_2\sqrt{t})^2}{2},\\
           \frac{r(x)^2}{2} \le &tf(\Phi_t(x))\le\frac{r(x)^2}{2}+Ct.
        \end{split}
    \end{equation*}
\end{corollary}
\begin{proof}
    Since $\Phi_t(p)=p$ for all $t>0$, then the proof follows immediately from [Lemma \ref{proper function},\eqref{eq: d and f}]. Moreover, since we have $d_{g}(p,\cdot)\le Af$ and $f\ge\varepsilon>0$ on $M$, we conclude that there exists a uniform constant $A>1$ such that
    \begin{equation}\label{eq: tft is aprabolic distance}
     A^{-1}(d_{g(t)}(p,x)+\sqrt{t})^2 \le  tf(\Phi_t(x))\le A(d_{g(t)}(p,x)+\sqrt{t})^2,
    \end{equation}
    holds for all $x\in M,t>0$. In other words, function $\sqrt{t\Phi_t^*f}$ is comparable to the \emph{parabolic distance} $d_{g(t)}(p,\cdot)+\sqrt{t}$.
\end{proof}
\begin{proof}[Proof of Lemma \ref{isometry lemma}]
    Let $g_i=\varrho(g_{\mathcal{C},i})\in \varrho(\mathcal{B}(g_\mathcal{C},\delta))$ with $g_{\mathcal{C},i}\in \mathcal{B}(g_\mathcal{C},\delta)$ for $i=1,2$, it suffices to show that $\operatorname{BL}_{\mathcal{C}}(g_{\mathcal{C},1},g_{\mathcal{C},2})= \operatorname{BL}_\infty(g_1,g_2)$. Since $g_{\mathcal{C},i}$ is conical metric, then one can apply Perelman's pseudolocality \cite[Appendix A]{ChenHallgrenLucas} to show there exists a $\lambda>0$ such that on $\{r^2\ge\lambda\}$
    \begin{equation*}
        |\pi_*g_i-g_{\mathcal{C},i}|_{g_{\mathcal{C},i}}\le C r^{-2},\quad i=1,2.
    \end{equation*}
   Recall that $d\pi(X)=r\partial_r$ and $(M,g_i,X)$ is a K\"ahler-Ricci expander, it follows that on $\{r^2\ge\lambda t\}$, for the self-similar solutions $g_i(t):=t\Phi_t^*g_i$, we have
    \begin{equation*}
          |\pi_*g_i(t)-g_{\mathcal{C},i}|_{g_{\mathcal{C},i}}\le C r^{-2}t,\quad i=1,2.
    \end{equation*}
    For any $\varepsilon>0$, we can find a $\lambda_\varepsilon>0$ such that for all $x\in\mathcal{C}$ with $r(x)^2\ge \lambda_\varepsilon$, it holds that
    \begin{equation*}
        \operatorname{BL}_x(\pi_*g_1,\pi_*g_2)\le\operatorname{BL}_\infty(g_1,g_2)+\varepsilon.
    \end{equation*}
    Therefore, for any $t>0$ and $x\in\mathcal{C}$ such that $r(x)^2\ge \lambda_\varepsilon t$, we have
    \begin{equation*}
        \operatorname{BL}_x(\pi_*g_1(t),\pi_*g_2(t))\le\operatorname{BL}_\infty(g_1,g_2)+\varepsilon.
    \end{equation*}
  Let $t,\varepsilon$ tend to 0, on $\mathcal{C}\setminus\{o\}$, one has
    \begin{equation*}
         \operatorname{BL}_{\mathcal{C}}(g_{\mathcal{C},1},g_{\mathcal{C},2})\le \operatorname{BL}_\infty(g_1,g_2)
    \end{equation*}
    as expected.

    For another side, we observe that
    \begin{equation*}
        \operatorname{BL}_\infty(g_i,\pi^*g_{\mathcal{C},i})=0,\quad \textnormal{for $i=1,2$}.
    \end{equation*}
    It follows that
    \begin{equation*}
         \operatorname{BL}_\infty(g_1,g_2)\le \operatorname{BL}_{\infty}(\pi^*g_{\mathcal{C},1},\pi^*g_{\mathcal{C},2})+\sum_{i=1}^2 \operatorname{BL}_\infty(g_i,\pi^*g_{\mathcal{C},i})\le \operatorname{BL}_{\infty}(\pi^*g_{\mathcal{C},1},\pi^*g_{\mathcal{C},2}).
    \end{equation*}
    By definition, $\operatorname{BL}_{\infty}(\pi^*g_{\mathcal{C},1},\pi^*g_{\mathcal{C},2})\le\operatorname{BL}_{\mathcal{C}}(g_{\mathcal{C},1},g_{\mathcal{C}_2})$, therefore,
    \begin{equation*}
         \operatorname{BL}_\infty(g_1,g_2)\le \operatorname{BL}_{\mathcal{C}}(g_{\mathcal{C},1},g_{\mathcal{C}_2}),
    \end{equation*}
    holds as expected.
\end{proof}
\section{Estimation along K\"ahler-Ricci flow}\label{estimation of KRF}

In this section, we prove Theorem \ref{main estimate thm}. We consider the initial K\"ahler metric $g_0$ on an asymptotically conical gradient K\"ahler-Ricci expander $(M,g,X)$ that satisfies the following conditions:
\begin{definition}[$a-$closeness]\label{a closeness}
We say that a K\"ahler metric $g_0$ is $a-$close to $\omega$ if it satisfies the following conditions:
    \begin{enumerate}
    \item Its Kähler form can be written as
\begin{equation*}
    \omega_0 = \omega + i\partial\bar{\partial}\psi_0,
\end{equation*}
for some smooth real-valued function $\psi_0$.
The bi-Lipschitz norm at infinity (see Definition \ref{bi-lipschitz norm at infinity}) is strictly bounded from above by the constant $a>0$:
    \begin{equation*}
        \operatorname{BL}_\infty(g,g_{0})<a.
    \end{equation*}
        \item Let $d_g(p,\cdot)$ denote the distance function for fixed point $p\in M$. For each $k\in\N_0$, there exists a constant $C_k>0$ such that 
        \begin{equation*}
           \sup_M d_g(p,\cdot)^k|(\nabla^g)^kg_0|_g\le C_k.
        \end{equation*}
        \item The metric $g_0$ is $JX-$invariant.
    \end{enumerate}
\end{definition}
\begin{definition}[bi-Lipschitz norm at infinity]\label{bi-lipschitz norm at infinity}
    For each $x\in M$, we define the bi-Lipschitz norm at $x$ as follows:
    \begin{equation*}
        \operatorname{BL}_x(g,g_0):=\sup\{\sigma\ge0\ |\ e^{-\sigma}g_0(x)\le g(x)\le e^\sigma g_0(x)\}.
    \end{equation*}
    Then we define bi-Lipschitz norm at infinity in the following way:
    \begin{equation*}
        \operatorname{BL}_\infty(g,g_{0}):=\limsup_{d_g(p,x)\to\infty} \operatorname{BL}_x(g,g_0).
    \end{equation*}
\end{definition}
\subsection{Quantification of initial metric}
Combining the metric equivalence with the 
$JX$-invariance, one can derive $C^0$ and $C^1$ estimates for the Kähler potential.
\begin{prop}[Quantification of initial data]\label{quantification of initial data}
    There exists a real-valued $JX$-invariant function $\psi_0\in C^\infty(M)$ such that the following holds:
    \begin{enumerate}
        \item The function $\psi_0$ is the K\"ahler potential of $\omega_0$, i.e, $\omega_0=\omega+i\partial\bar\partial\psi_0$.
        \item Let $f$ be the normalized soliton potential, there exists a constant $C>0$ such that on $M$,
        \begin{equation*}
            \begin{split}
               & (e^{-a}-1)f-C\le\psi_0\le (e^a-1)f+C;\\
               & (e^{-a}-1)f-C\le\frac{X}{2}\cdot\psi_0\le (e^a-1)f+C.
            \end{split}
        \end{equation*}
    \end{enumerate}
\end{prop}
\begin{proof}
Let $\psi_0$ be the K\"ahler potential of $\omega_0$ in Definition \ref{a closeness}.
    We can always assume that $JX\cdot\psi_0=0$ in Definition \ref{a closeness} thanks to \cite[Proposition 3.7]{longteng2}.

    By the condition of bi-Lipschitz norm at infinity, there exist $\lambda_0,\varepsilon>0$ such that on $\{f\ge\lambda_0\}$, we have
    \begin{equation*}
       e^{-a+\varepsilon} g_0\le g\le e^{a-\varepsilon}g_0.
    \end{equation*}
    Here we can replace $d^2_g(p,\cdot)$ by $f(\cdot)$ thanks to Lemma \ref{proper function}.
    Applying this inequality to $X$ gives us
    \begin{equation*}
      e^{-a+\varepsilon}g(X,X) \le g_0(X,X)\le e^{a-\varepsilon}g(X,X).
    \end{equation*}
    We conclude that $\partial\bar\partial\psi_0(X,X)\le (e^{a-\varepsilon}-1)g(X,X)$. Since $\psi_0$ is $JX-$invariant, then we have 
    \begin{equation*}
        \partial\bar\partial\psi_0(X,X)=X\cdot(\frac{X}{2}\cdot\psi_0)\le (e^{a-\varepsilon}-1)g(X,X)=(e^{a-\varepsilon}-1)X\cdot f.
    \end{equation*}
    Therefore, we have on $\{f\ge\lambda_0\}$
    \begin{equation*}
        \frac{X}{2}\cdot\psi_0-(e^{a-\varepsilon}-1)f\le \sup_{f(x)\le \lambda_0}|\frac{X}{2}\cdot\psi_0-(e^{a-\varepsilon}-1)f|(x).
    \end{equation*}
    Define $C_1=\sup_{f(x)\le \lambda_0}|\frac{X}{2}\cdot\psi_0-(e^{a-\varepsilon}-1)f|(x)$, then it follows that on $M$, we have
    \begin{equation*}
        \frac{X}{2}\cdot\psi_0-(e^{a-\varepsilon}-1)f\le C_1.
    \end{equation*}
    Similarly, we can prove that on $M$, we also have
    \begin{equation*}
         \frac{X}{2}\cdot\psi_0\ge (e^{-a+\varepsilon}-1)f- C_1.
    \end{equation*}
    Recall that on the K\"ahler cone $(\mathcal{C},g_\mathcal{C})$, by Lemma \ref{proper function}, we have $X=r\partial_r$ and
    \begin{equation*}
       \frac{r^2}{2}\le f\le \frac{r^2}{2}+C,
    \end{equation*}
    holds for some uniform constant $C>0$. For $\lambda>1$, on $\{r^2\ge \lambda\}$, we have
    \begin{equation*}
        r\partial_r\psi_0\le 2(e^{a-\varepsilon}-1)f+C_1\le (e^{a-\varepsilon}-1)r^2+2C+C_1.
    \end{equation*}
    By integration, we have
    \begin{equation*}
      \begin{split}
           \psi_0-\psi_0|_{r^2=\lambda}&\le (e^{a-\varepsilon}-1)\left(\frac{r^2}{2}-\frac{\lambda}{2}\right)+(2C+C_1)(\log r-\frac{1}{2}\log\lambda)\\
            &\le (e^{a-\varepsilon}-1)\frac{r^2}{2}+(2C+C_1)\log r.
      \end{split}
    \end{equation*}
Take $\lambda>>1$ such that $(2C+C_1)\log r\le \sigma\frac{r^2}{2}$ on $\{r^2\ge\lambda\}$. Here $\sigma>0=\min\{e^a-e^{a-\varepsilon},e^{-a+\varepsilon}-e^{-a}\}$. Then on $\{r^2\ge\lambda\}$, we have
\begin{equation*}
   \psi_0-\psi_0|_{r^2=\lambda}\le (e^a-1)\frac{r^2}{2}\le (e^a-1)f.
\end{equation*}
Similarly, we can also prove thta on $\{r^2\ge\lambda\}$, we have
\begin{equation*}
    \psi_0-\psi_0|_{r^2=\lambda}\ge (e^{-a}-1)\frac{r^2}{2}\ge (e^{-a}-1)f.
\end{equation*}
    Let $C=\sup_{r^2\le \lambda}|\psi_0|+C_1$, then on $M$, we have 
     \begin{equation*}
            \begin{split}
               & (e^{-a}-1)f-C\le\psi_0\le (e^a-1)f+C;\\
               & (e^{-a}-1)f-C\le\frac{X}{2}\cdot\psi_0\le (e^a-1)f+C.
            \end{split}
        \end{equation*}
\end{proof}
\subsection{Rough estimates along K\"ahler-Ricci flow}
Let $(\tilde g(t))_{t\in [0,T)}$ be the unique Shi's K\"ahler-Ricci flow starting from $g_0$ with maximum existence time $T>0$ as in \cite[Proposition 3.8]{longteng2}. Recall that $\operatorname{BL}_\infty(g_0,g)<a$ and $\operatorname{BL}_\infty(g_{\mathcal{C}},g)=0$, when the initial data is close enough to the conical metric, Perelman's pseudolocality together with Shi's estimates will control the flow:
\begin{prop}\label{perelman+shi}
    There exists an $a>0$, such that if $g_0$ is $a-$close to $g$,  then there exist $\{C_k>0\}_{k\in\N_0}$, $\lambda>0$, such that on $\{(x,t)\ |\ r(x)^2\ge \lambda+\lambda t,\ t\in[0,T)\}$, we have 
    \begin{equation*}
        r(x)^{2+k}|(\nabla^{\tilde g(t)})^k\Rm(\tilde g(t))|_{\tilde g(t)}(x)\le C_k.
    \end{equation*}
\end{prop}
\begin{proof}
   The proof is the same as in \cite[Appendix A]{ChenHallgrenLucas}.
\end{proof}
\begin{corollary}
    There exist $\{C'_k>0\}_{k\in\N_0}$, such that on $\{(x,t)\ |\ r(x)^2\ge \lambda+\lambda t,\ t\in[0,T)\}$, we have 
    \begin{equation*}
        r(x)^{k}|(\nabla^{g(t+1)})^k\left(\tilde g(t)-g(t+1)\right)|_{g(t+1)}(x)\le C'_k.
    \end{equation*}
\end{corollary}
\begin{proof}
    The proof follows by the same argument as in the proof of \cite[Claim 3.9]{ConlonDeruelleSun}.
\end{proof}
Throughout the remainder of this article, we assume that the initial metric $g_0$ is 
$a$-close to $g$ in the sense of Proposition \ref{perelman+shi}. To perform estimates on the region $\{(x,t)|r(x)^2\le\lambda(t+1)\}$, we need to reduce the K\"ahler-Ricci flow equation to a Monge-Ampère flow equation:
\begin{prop}\cite[Proposition 3.10]{longteng2}\label{MA flow}
    There exists an $JX-$invariant real-valued function $\varphi\in C^\infty(M\times [0,T))$ such that $\tilde\omega(t)=\omega(1+t)+i\partial\bar\partial \varphi(t)$ for all $t\in [0,T)$. Moreover, $\varphi(t)_{t\in [0,T)}$ is the solution to the following Monge-Ampère flow equation:
    \begin{equation*}
        \begin{split}
            \frac{\partial}{\partial t}\varphi(t)&=\log\frac{(\omega(1+t)+i\partial\bar\partial \varphi(t))^n}{\omega(1+t)^n};\\
            \varphi(0)&=\psi_0.
        \end{split}
    \end{equation*}
\end{prop}
In this section, we shall also denote $\tilde g(t)$(resp. $\tilde \omega(t)$) by $g_\varphi(t)$(resp.$\omega_\varphi(t)$).
\begin{prop}\label{rough estimates}
    We can take $\lambda>>1$ sufficiently large such that on $\{(x,t)\ |\ r(x)^2\ge\lambda(t+1),\ t\in[0,T)\}$, we have that
    \begin{enumerate}
        \item $e^{-a}g(t+1)\le g_\varphi(t)\le e^ag(t+1)$;
        \item $(e^{-a}-1)f-C(t+1)\le\varphi(t)\le (e^a-1)f+C(t+1)$;
        \item $(e^{-a}-1)f-C(t+1)\le\frac{X}{2}\cdot\varphi(t)\le (e^a-1)f+C(t+1)$.
    \end{enumerate}
\end{prop}
\begin{proof}
    First take $\lambda>>1$ such that on $\{r^2\ge\lambda\}$, there is an $\varepsilon>0$ such that
    \begin{equation*}
       e^{-a+\varepsilon}g \le g_0\le e^{a-\varepsilon}g.   
    \end{equation*}
Since on $\{r^2\ge\lambda(t+1)\}$, we have $|\Rm(g_\varphi(t))|_{g_\varphi(t)}\le \frac{C}{r^2}$, it follows that
\begin{equation*}
    e^{-\frac{Ct}{r^2}}g_0\le g_\varphi(t)\le e^{\frac{Ct}{r^2}}g_0.
\end{equation*}
Similarly, we have
\begin{equation*}
    e^{-\frac{Ct}{r^2}}g\le g(t+1)\le e^{\frac{Ct}{r^2}}g.
\end{equation*}
Combine these, we get
\begin{equation*}
   \begin{split}
        &g_\varphi(t)\le e^{a-\varepsilon+\frac{2Ct}{r^2}}g\le e^{a-\varepsilon+\frac{2C}{\lambda}}g;\\
        &g_\varphi(t)\ge e^{-a+\varepsilon-\frac{2Ct}{r^2}}g(t+1)\le e^{-a+\varepsilon-\frac{2C}{\lambda}}g.
   \end{split}
\end{equation*}
Take $\lambda>>1$ such that $\varepsilon-\frac{2C}{\lambda}>0$, then we have 
\begin{equation*}
    e^{-a}g(t+1)\le g_\varphi(t)\le e^ag(t+1),
\end{equation*}
holds on $\{r^2\ge\lambda(t+1)\}$.
   
    Recall the equation of Monge-Amp\`ere flow:
    \begin{equation*}
        \frac{\partial}{\partial t}\varphi=\log\frac{\omega_\varphi(t)^n}{\omega(t+1)^n}.
    \end{equation*}
    By integrating, we get
    \begin{equation*}
        |\varphi(t)-\varphi(0)|\le na t.
    \end{equation*}
    Recall that $\varphi(0)=\psi_0\le (e^a-1)f+C$ as in Proposition \ref{quantification of initial data}, then we have
    \begin{equation*}
         \varphi(t)\le (e^a-1)f+C+nat.
    \end{equation*}
Similarly, it follows that
\begin{equation*}
    \varphi(t)\ge (e^{-a}-1)f-C-nat.
\end{equation*}
    We also have
    \begin{equation*}
        \begin{split}
            \frac{X}{2}\cdot\frac{\partial}{\partial t}\varphi&=\tr_{\omega_\varphi(t)}\mathcal{L}_{\frac{X}{2}}\omega_\varphi(t)-\tr_{\omega(t)}\mathcal{L}_{\frac{X}{2}}\omega(t+1)\\
            &=\tr_{\omega_\varphi(t)}\mathcal{L}_{\frac{X}{2}}\omega(t+1)-\tr_{\omega(t)}\mathcal{L}_{\frac{X}{2}}\omega(t+1)+\tr_{\omega_\varphi(t)}\mathcal{L}_{\frac{X}{2}}i\partial\bar\partial\varphi(t).
        \end{split}
    \end{equation*}
    Since we can control $\tr_{\omega_\varphi(t)}\mathcal{L}_{\frac{X}{2}}i\partial\bar\partial\varphi(t)$ in the following way: \begin{equation*}
    \begin{split}
            |\tr_{\omega_\varphi(t)}\mathcal{L}_{\frac{X}{2}}i\partial\bar\partial\varphi(t)|&\le C(n)|g_\varphi(t)|_{g(t+1)}\left(|X|_{g(t)}|\nabla^{g(t+1)}i\partial\bar\partial\varphi(t)|_{g(t+1)}+|\nabla^{g(t+1)}X|_{g(t+1)}|i\partial\bar\partial\varphi(t)|_{g(t+1)}\right)\\
            &\le C.
    \end{split}
    \end{equation*}
    Therefore, $ |\frac{X}{2}\cdot\frac{\partial}{\partial t}\varphi|$ is uniformly bounded by some constant $C_1>0$ on $\{r^2\ge\lambda(t+1)\}$. By integration, we have
    \begin{equation*}
         |\frac{X}{2}\cdot\varphi(t)-\frac{X}{2}\cdot\varphi(0)|\le C_1t.
    \end{equation*}
    Recall that $\frac{X}{2}\cdot\varphi(0)=\frac{X}{2}\cdot\psi_0\le (e^a-1)f+C$. Then we have
    \begin{equation*}
        \frac{X}{2}\cdot\varphi(t)\le (e^a-1)f+C+C_1t;\quad \frac{X}{2}\cdot\varphi(t)\ge (e^{-a}-1)f-C-C_1t.
    \end{equation*}
    Take $C=C+C_1+na$, we have
    \begin{equation*}
        \begin{split}
           & (e^{-a}-1)f-C(t+1)\le \varphi(t)\le (e^a-1)f+C(t+1);\\
           &(e^{-a}-1)f-C(t+1)\le\frac{X}{2}\cdot\varphi(t)\le (e^a-1)f+C(t+1).  
        \end{split}
          \end{equation*}
\end{proof}
\subsection{Uniform estimates along normalized K\"ahler-Ricci flow}
In this section, we will establish some uniform estimates along the normalized K\"ahler-Ricci flow. We first introduce the self-similar variables:
\begin{definition}[Self-similar variables]\label{self-similar variables}
    Let $\tilde g(t)_{t\in [0,T)}$ be a solution to K\"ahler-Ricci flow, the corresponding normalized K\"ahler-Ricci flow in self-similar variables $\bar g(\tau)_{\tau\in [0,\log(1+T))}$ is defined as follows:
    \begin{equation*}
        \bar g(\tau)=e^{-\tau}\Phi_{e^{-\tau}}^*\tilde g(e^\tau-1).
    \end{equation*}
    Moreover, $\bar g(\tau)$ is a solution to \emph{the normalized K\"ahler-Ricci flow} equation:
    \begin{equation*}
        \frac{\partial}{\partial\tau}\bar g(\tau)=\mathcal{L}_{\frac{X}{2}}\bar g(\tau)-\Ric(\bar g(\tau))-\bar g(\tau),\quad \tau\in [0,\log(1+T)).
    \end{equation*}
\end{definition}
Let $\varphi(t)_{t\in [0,T)}$ be the solution to Monge-Amp\`ere flow in Proposition \ref{MA flow}, we define $\psi(\tau):=e^{-\tau}\Phi_{e^{-\tau}}^*\varphi(e^\tau-1)$ for all $\tau\in [0,\log(T+1)$, then $\psi$ is a $JX-$invariant solution to the \emph{normalized Monge-Amp\`ere flow}:
\begin{equation*}
   \begin{split}
        \frac{\partial}{\partial\tau}\psi(\tau)&=\log\frac{(\omega+i\partial\bar\partial\psi(\tau))^2}{\omega^n}+\frac{X}{2}\cdot\psi(\tau)-\psi(\tau);\\
        \psi(0)&=\psi_0.
   \end{split}
\end{equation*}
And we have immediately that $e^{-\tau}\Phi_{e^{-\tau}}^*\tilde \omega(e^\tau-1)=\omega+i\partial\bar\partial\psi(\tau)$. We use $\omega_\psi(\tau)$ to denote $\omega+i\partial\bar\partial\psi(\tau)$ and $g_\psi(\tau)$ denotes its Riemannian metric. From Proposition \ref{rough estimates}, we state:
\begin{prop}\label{boundary data}
    On $\{r^2\ge \lambda\}\times [0,\log(1+T))$, there exists a constant $C>0$ such that
    \begin{enumerate}
    \item\label{prop1} $e^{-a}g\le g_\psi\le e^ag$;
        \item\label{prop2} $(e^{-a}-1)f-C\le \psi\le (e^a-1)f+C$;
        \item\label{prop3} $(e^{-a}-1)f-C\le \frac{X}{2}\cdot\psi\le (e^a-1)f+C$.
    \end{enumerate}
       Moreover, (ii) and (iii) also hold on the time slice $M\times\{0\}$.
\end{prop}
\begin{proof}
   \ref{prop1}-\ref{prop3} follow directly from Proposition \ref{quantification of initial data} and \ref{rough estimates} and the fact that the function $t\Phi_t^*f$ is an increasing function on $t>0$. As a consequence, by recalling the normalized Monge-Ampère flow, we have the following control:
    \begin{equation*}
        |\dot\psi|=|\log\frac{\omega_\psi^n}{\omega^n}+\frac{X}{2}\cdot\psi-\psi|\le (e^a-e^{-a})f+2C+na.
    \end{equation*}
\end{proof}
\begin{prop}
    The function $f_\psi:=f+\frac{X}{2}\cdot\psi$ is a Hamiltonian of vector field $X$ with respect to metric $g_\psi$. Along the flow, it holds that on $\{r^2\ge\lambda\}\times [0,\log(1+T))$
    \begin{equation*}
      e^{-a}f-C  \le f_\psi\le e^a f+C.
    \end{equation*}
    Moreover, the above inequality also holds on time slice $M\times \{0\}$.
    
    In addition, on $M,$ we have
    \begin{equation*}
        \inf_M f_\psi\ge \inf_M f\ge \varepsilon>0.
    \end{equation*}
\end{prop}
\begin{proof}
The function $f_\psi$ is a Hamiltonian of vector field $X$ is due to the fact that $JX\cdot\psi=0$(see \cite[Lemma 4.12]{ChenHallgrenLucas}).

    Indeed, from the previous proof, one can clearly see that along the normalized K\"ahler-Ricci flow, we actually have
    \begin{equation*}
        \frac{X}{2}\cdot\psi\ge (e^{-a}-1)f-C.
    \end{equation*}
    It follows that $f_\psi=f+\frac{X}{2}\cdot\psi\ge e^{-a}f-C$. This fact implies that $f_\psi$ is a proper function. Suppose that $x\in M$ is a minimum point of $f_\psi$, then necessarily we must have
    \begin{equation*}
        X(x)=\nabla^{g_\psi}f_\psi(x)=0.
    \end{equation*}
Then we have
\begin{equation*}
    f_\psi(x)=f(x)\ge\inf_M f\ge \varepsilon>0.
\end{equation*}
\end{proof}
With the boundary data provided by Proposition \ref{boundary data}, we can apply the maximum principle, as in \cite[Section 4]{ChenHallgrenLucas}, on the region ${r^2 \le \lambda} \times [0, \log(T+1))$. This yields the following results:

\begin{corollary}\label{basic estimates}
    On $\{r^2\le\lambda\}\times [0,\log(T+1))$, we have
    \begin{equation*}
        \begin{split}
            &-(e^a-1)f_\psi-C\le \psi\le (e^a-1)f+C;\\
            &|\dot\psi|\le (e^{2a}-1)f_\psi+C;\quad |X|_{g_\psi}^2\le Cf_{\psi}+C.
        \end{split}
    \end{equation*}
    holds for some uniform constant $C>0$.
\end{corollary}
 The following theorem, based on Yau’s $C^2$-estimate, implies that if the two potentials $f_\psi$ and $f$ are comparable, then the normalized flow remains uniformly bi-Lipschitz equivalent to the soliton metric.

\begin{theorem}[$C^2$-estimate]
    If along the normalized K\"ahler-Ricci flow $\omega_\psi(\tau)_{\tau\in [0,\log (T+1))}$, the following inequality
    \begin{equation*}
        \frac{1}{D}f\le f_\psi\le Df,
    \end{equation*}
    holds for some uniform constant $D>1$. Then there exists a uniform constant $C>1$ such that along the flow, we have
    \begin{equation*}
        \frac{1}{C}g\le g_\psi\le Cg.
    \end{equation*}
\end{theorem}
\begin{proof}
    The proof is the same as in \cite[Theorem 4.15]{ChenHallgrenLucas}.
\end{proof}
The following proposition shows that if the initial data $g_0$ is sufficiently \emph{close} to $g$, then the two potentials are \emph{comparable}.

\begin{prop}
    There exists an $a>0$, such that if $g_0$ is $a-$close to $g$, then there exists a uniform constant $D>1$ such that along the flow, we have
    \begin{equation*}
        \frac{1}{D}f\le f_\psi\le Df
    \end{equation*}
\end{prop}
\begin{proof}
    Recall that $f_\psi$ satisfies that following evolution equation (\cite[Lemma 4.4]{ChenHallgrenLucas}):
    \begin{equation*}
       \left( \frac{\partial}{\partial t}-\Delta_{\omega_\psi,X}\right)f_\psi=-f_\psi.
    \end{equation*}
    Here $\Delta_{\omega_\psi,X}:=\Delta_{\omega_\psi}+\frac{X}{2}$.
    
    Let $A_{\pm}\ge0$ be defined as follows:
    \begin{equation*}
       \begin{split}
            &A_+:=\inf\{K\ge0\ |\  (\nabla^g)^2f-Kg\le0\};\\
            &A_-:=\inf\{K\ge0\ | \ (\nabla^g)^2f+Kg\ge0\}.
       \end{split}
    \end{equation*}
    Since $(M,g,X)$ is an asymptotically conical Ricci expander, the constants $A_+, A_-$ are well defined and nonnegative. On $M$, we have $-A_-g\le(\nabla^g)^2f\le A_+g$. By soliton equation $\Ric(g)+g=(\nabla^g)^2f$ and the fact that $R_\omega+n$ is strictly positive, we conclude that $A_+>0$.

    Take $a>0$ such that $A_+(e^{2a}-1)+A_+^2(e^a-1)<A_+$ and $A_-(e^a-1)<1$. We only need to show on $\{r^2\le\lambda\}\times[0,\log(T+1))$, 
    \begin{equation*}
        \frac{1}{D}f\le f_\psi\le Df,
    \end{equation*}
    holds for some uniform constant $D>1$.

    First, we compute
    \begin{equation*}
       \begin{split}
            \left(\frac{\partial}{\partial\tau}-\Delta_{\omega_\psi,X}\right)(f-A_-\psi)&=-\Delta_{\omega_\psi}f-\frac{X}{2}\cdot f-A_-\dot\psi+A_-\Delta_{\omega_\psi}\psi+A_-\frac{X}{2}\cdot\psi\\
            &=-\Delta_{\omega_\psi}f-|\partial f|_g^2-A_-\dot\psi-A_-(\tr_{\omega_\psi}\omega-n)+A_-f_\psi-A_-f\\
            &\le C_0-\Delta_{\omega_\psi}f-A_-\dot\psi-A_-(\tr_{\omega_\psi}\omega-n)A_-f_\psi-(A_-+1)f.
       \end{split}
    \end{equation*}
   Here $C_0\ge 0$ such that $|\partial f|_g^2\ge f-C_0$. Since $-A_-g\le(\nabla^g)^2f$, it follows that
    \begin{equation*}
        \begin{split}
            \left(\frac{\partial}{\partial\tau}-\Delta_{\omega_\psi,X}\right)(f-A_-\psi)&\le C_0+A_-\tr_{\omega_\psi}\omega-A_-\dot\psi-A_-(\tr_{\omega_\psi}\omega-n)+A_-f_\psi-(A_-+1)f\\
            &=A_-n+C_0-A_-\dot\psi+A_-f_\psi-(A_-+1)f\\
            &=A_-n+C_0-A_-\dot\psi+A_-f_\psi-(A_-+1)(f-A_-\psi)-A_-(A_-+1)\psi.
        \end{split}
    \end{equation*}
    The Corollary \ref{basic estimates} shows that $\dot\psi\ge -(e^{2a}-1)f_\psi-C$ and $\psi\ge -(e^a-1)f_\psi-C$, then we have
    \begin{equation*}
        \begin{split}
            \left(\frac{\partial}{\partial\tau}-\Delta_{\omega_\psi,X}\right)(f-A_-\psi)&\le 
            C_1+C_2f_\psi-(A_-+1)(f-A_-\psi).
        \end{split}
    \end{equation*}
    Here $C_1=C_0+A_-n+A_-C+A_-^2C,\ C_2=(A_-+A_-(e^{2a}-1)+A_-(A_-+1)(e^a-1))$. Now we consider $u=f-A_-\psi-Bf_\psi$ with $B>0$ to be determined. We compute
    \begin{equation*}
        \begin{split}
            \left(\frac{\partial}{\partial\tau}-\Delta_{\omega_\psi,X}\right)u&\le Bf_\psi +C_1+C_2f_\psi-(A_-+1)(f-A_-\psi)\\
            &=Bf_\psi +C_1+C_2f_\psi-(A_-+1)(f-A_-\psi-Bf_\psi)-(A_-+1)Bf_\psi\\
            &=C_1+(C_2-A_-B)f_\psi-(A_-+1)u.
        \end{split}
    \end{equation*}
    One can take $B>>1$ such that $C_2-A_-B<0$ and on $\{r^2=\lambda\}\bigcup M\times\{0\}$, $u=f-A_-\psi-Bf_\psi\le C_3$ holds for some uniform constant $C_3$. Then it follows that $ \left(\frac{\partial}{\partial\tau}-\Delta_{\omega_\psi,X}\right)u\le C_1-(A_-+1)u$, by maximum principle, we have on $\{r^2\le\lambda\}\times[0,\log(t+1))$,
    \begin{equation*}
        f-A_-\psi-Bf_\psi=u\le C_4,
    \end{equation*}
    holds for some uniform constant $C_4>0$. Since $A_-\psi\le A_-(e^a-1)f+A_-C$ and we have taken $a>0$ such that $A_-(e^a-1)<1$, therefore
    \begin{equation*}
        f\le C_4+A_-(e^a-1)f+A_-C+Bf_\psi.
    \end{equation*}
    It follows that
    \begin{equation*}
f\le \frac{1}{1-A_-(e^a-1)}\left(C_4+A_-C+Bf_\psi\right)\le \frac{1}{1-A_-(e^a-1)}\left(\frac{C_4+A_-C}{\varepsilon}+B\right)f_\psi.
    \end{equation*}
Here $\varepsilon=\inf_M f>0$.

    Second, we compute
    \begin{equation*}
         \begin{split}
             \left(\frac{\partial}{\partial\tau}-\Delta_{\omega_\psi}+\frac{X}{2}\right)(f+A_+\psi)&=-\Delta_{\omega_\psi}f+|\partial f|_g^2+A_+\dot\psi-A_+\Delta_{\omega_\psi}\psi+A_+\frac{X}{2}\cdot\psi\\
             &=-\Delta_{\omega_\psi}f+|\partial f|_g^2+A_+\dot\psi-A_+\Delta_{\omega_\psi}\psi+A_+f_\psi-A_+f\\
             &\ge -\Delta_{\omega_\psi}f+A_+\dot\psi-A_+\Delta_{\omega_\psi}\psi+A_+f_\psi-A_+f.
         \end{split}
    \end{equation*}
    Since $(\nabla^g)^2f\le A_+g$, then we have
    \begin{equation*}
        \begin{split}
             \left(\frac{\partial}{\partial\tau}-\Delta_{\omega_\psi}+\frac{X}{2}\right)(f+A_+\psi)&\ge -A_+\tr_{\omega_\psi}\omega+A_+\dot\psi-A_+\Delta_{\omega_\psi}\psi+A_+f_\psi-A_+f\\
             &=-A_+n+A_+\dot\psi+A_+f_\psi-A_+f\\
             &=-A_+n+A_+\dot\psi+A_+f_\psi-A_+(f+A_+\psi)+A_+^2\psi.
         \end{split}
    \end{equation*}
     The Corollary \ref{basic estimates} shows that $\dot\psi\ge -(e^{2a}-1)f_\psi-C$ and $\psi\ge -(e^a-1)f_\psi-C$, then we have
     \begin{equation*}
         \begin{split}
            \left(\frac{\partial}{\partial\tau}-\Delta_{\omega_\psi}+\frac{X}{2}\right)(f+A_+\psi)&\ge   -C_5+\sigma f_\psi-A_+(f+A_+\psi).
         \end{split}
     \end{equation*}
     Here $C_5=A_+n+A_+^2C+A_+C$, $\sigma=\left(A_+-A_+(e^{2a}-1)-A_+^2(e^a-1)\right)$. By the definition of $a$, constant $\sigma>0$. Now we consider $v=f+A_+\psi-\alpha f_\psi$ with $\alpha>0$ to be determined. We compute
     \begin{equation*}
         \begin{split}
            \left(\frac{\partial}{\partial\tau}-\Delta_{\omega_\psi}+\frac{X}{2}\right)v&\ge   -C_5+\sigma f_\psi-A_+(f+A_+\psi)-\alpha(X\cdot f_\psi-f_\psi)\\
            &\ge -C_5+\sigma f_\psi-A_+(f+A_+\psi)-\alpha |X|_{g_\psi}^2\\
            &=-C_5+\sigma f_\psi-A_+(f+A_+\psi-\alpha f_\psi)-A_+\alpha f_\psi-\alpha |X|_{g_\psi}^2.
         \end{split}
     \end{equation*}
     Corollary \ref{basic estimates} tells us that $|X|_{g_\psi}^2\le Cf_\psi+C$, it follows that
     \begin{equation*}
         \begin{split}
            \left(\frac{\partial}{\partial\tau}-\Delta_{\omega_\psi}+\frac{X}{2}\right)v\ge-C_5-\alpha C+(\sigma -\alpha C-A_+\alpha)f_\psi-A_+v.
         \end{split}
     \end{equation*}
     One can take $\alpha<<1$ such that $\sigma -\alpha C-A_+\alpha\ge 0$ and $(1-A_+(e^a-1))-\alpha\ge0$, therefore, on $\{r^2=\lambda\}\bigcup M\times\{0\}$, 
     \begin{equation*}
         v=f+A_+\psi-\alpha f_\psi\ge f+A_+(e^{-a}-1)f-A_+C-\alpha f_\psi\ge -A_+C+ (1-A_+(e^a-1))f-\alpha f_\psi>-C_6
     \end{equation*}holds for some uniform constant $C_6>0$. Then it follows that $ \left(\frac{\partial}{\partial\tau}-\Delta_{\omega_\psi,X}\right)v\ge -C_5-\alpha C-A_+v$, by maximum principle, we have on $\{r^2\le\lambda\}\times[0,\log(t+1))$,
     \begin{equation*}
         f+A_+\psi-\alpha f_\psi=v\ge -C_7,
     \end{equation*}
     holds for some uniform constant $C_7>0$. Then we have
     \begin{equation*}
         f_\psi\le\frac{1}{\alpha}\left(f+A_+\psi+C_7\right)\le \frac{1}{\alpha}\left(f+A(e^a-1)f+AC+C_7\right)\le C_8f.
     \end{equation*}
     Here $C_8=\frac{1}{\alpha}\left(1+A_+(e^a-1)+\frac{C_7+A_+C}{\varepsilon}\right)$. Take $D=\max\{C_8,\frac{1}{1-A_-(e^a-1)}\left(\frac{C_4+A_-C}{\varepsilon}+B\right)\}$, we get
     \begin{equation*}
         \frac{1}{D}f\le f_\psi\le Df,
     \end{equation*}
     holds on $\{r^2\le\lambda\}\times[0,\log(T+1))$.
\end{proof}
The high order estimates follow naturally as in \cite[Appendix C]{ChenHallgrenLucas}
\begin{prop}\label{final results of estimates}
     Along the flow, for all $k\in\N_0$, there exists a uniform constant $C_k>0$ such that on $M\times [0,\log(1+T))$,
    \begin{equation*}
       f^{\frac{k}{2}}|(\nabla^{g})^k g_\psi|_{g}+  f^{1+\frac{k}{2}}|(\nabla^{g_\psi})^k\Rm(g_\psi)|_{g_\psi}\le C_k.
    \end{equation*}
\end{prop}
\begin{proof}[Proof of Theorem \ref{main estimate thm}]
    The proof of Theorem \ref{main estimate thm} comes directly from Proposition \ref{final results of estimates} and Hamilton's criterion on the maximum existence time.
\end{proof}
\section{Long-time behaviors of flows and asymptotic behaviors of initial values}\label{Long-time behaviors of flows and asymptotic behaviors of initial values}
In this section, we relate the $\varkappa-$limit set and $\kappa-$limit set of a K\"ahler metric $g_0$ that is $a-$close to $g$ as in Theorem \ref{main estimate thm}. 
\begin{proof}[Proof of Theorem \ref{relation between two limit sets}]
Let $\tilde g(\cdot)$ be the unique Shi's solution to K\"ahler-Ricci flow with initial data $g_0$.
    \begin{claim}\label{claim 1}
          If $g_0$ is $a-$close to $g$ as in Theorem \ref{main estimate thm}, then $\varkappa(g_0)$ and $\kappa(g_0)$ are nonempty.
    \end{claim}
    \begin{proof}
       We first show that $\varkappa(g_0)$ is not empty. Recall that $g_0$ is $a-$close to $g$, then the following holds for $p\in E$:
       \begin{enumerate}
           \item There exists a uniform constant $C>1$ such that along the flow, we have
        \begin{equation*}
           C^{-1}g(t+1)\le \tilde g(t)\le Cg(t+1).
        \end{equation*}
        \item For each $k\in \N_0$, there exists a $C_k>0$ such that along the flow, we have
        \begin{equation*}
           \left(\sqrt{t+1} +d_{g(t+1)}(p,\cdot)\right)^k|(\nabla^{g(t+1)})^k\tilde g(t)|_{g(t+1)}\le C_k.
        \end{equation*}
       \end{enumerate}
       Therefore, for all $\lambda>0$, we have
       \begin{enumerate}
             \item There exists a uniform constant $C>1$ such that along the flow, we have  \begin{equation*}
           C^{-1}g(t+\lambda)\le \mathcal{D}(\lambda)\tilde g(t)\le Cg(t+\lambda).
        \end{equation*}
        \item For each $k\in \N_0$, there exists a $C_k>0$ such that along the flow, we have
        \begin{equation*}
             \left(\sqrt{t+\lambda} +d_{g(t+\lambda)}(p,\cdot)\right)^k|(\nabla^{g(t+\lambda)})^k\mathcal{D}(\lambda)\tilde g(t)|_{g(t+\lambda)}\le C_k.
        \end{equation*}
       \end{enumerate}
       Then by Arzel\`a-Ascoli lemma, for any sequence $\lambda_n\to 0^+$, we can find a subsequence of it (which we still denote by $\{\lambda_n\}_{n\in\N_0}$) such that $ \mathcal{D}(\lambda_n)\tilde g(t)_{t\in (0,\infty)}$ converges to a K\"ahler-Ricci flow $\tilde h(t)_{t\in (0,\infty)}$ locally smoothly on $M\times (0,\infty)$. Then it follows that $\tilde h(t)_{t\in (0,\infty)}\in\varkappa(g_0)$.

      Now we show that $\kappa(g_0)$ is not empty. Recall that from above we also have:
       \begin{enumerate}
           \item There exists a uniform constant $C>1$ such that along the flow, we have
        \begin{equation*}
           C^{-1}g\le g_0\le Cg.
        \end{equation*}
        \item For each $k\in \N_0$, there exists a $C_k>0$ such that along the flow, we have
        \begin{equation*}
            r^k|(\nabla^{g})^kg_0|_{g}\le C_k.
        \end{equation*}
       \end{enumerate}
       Thus, for all $\lambda>0$, we have
       \begin{enumerate}
           \item There exists a uniform constant $C>1$ such that along the flow, we have
        \begin{equation*}
           C^{-1}g(\lambda)\le\mathcal{D}'(\lambda) g_0\le Cg(\lambda).
        \end{equation*}
        \item For each $k\in \N_0$, there exists a $C_k>0$ such that along the flow, we have
        \begin{equation*}
            r^k|(\nabla^{g(\lambda)})^k\mathcal{D}'(\lambda) g_0|_{g(\lambda)}\le C_k.
        \end{equation*}
       \end{enumerate}
        Then by Arzel\`a-Ascoli lemma, for any sequence $\lambda_n\to 0^+$, we can find a subsequence of it (which we still denote by $\{\lambda_n\}_{n\in\N_0}$) such that $ \mathcal{D}'(\lambda_n)g_0$ converges to a metric $\tilde h_0$ locally smoothly on each annulus on $M\setminus E$. Then it follows that $\tilde h_0\in\kappa(g_0)$ as $\lambda_n$ tends to $0$.
    \end{proof}
    \begin{claim}\label{claim 2}
        Let $\tilde h(t)=\lim_{n\to \infty}\mathcal{D}(\lambda_n)\tilde g(t)$ and $\tilde h_0=\lim_{n\to \infty}\mathcal{D}'(\lambda_n)g_0$ for the same sequence $\{\lambda_n>0\}_{n\to\infty}$ that converges to 0. Then $\tilde h(t)$ converges to $\tilde h_0$ on $M\setminus E$ locally smoothly.
    \end{claim}
    \begin{proof}
        We first fix an annulus $\{r_0^2\le r^2\le R_0^2\}$ on $\mathcal{C}$ and we define $K:=\pi^{-1}\{r_0^2\le r^2\le R_0^2\}$ as a compact set on $M\setminus E$. We show that, as $t$ tends to 0, $\tilde h(t)$ converges to $\tilde h_0$ smoothly on $K$.

        For all $k\in\N_0,\varepsilon>0$ that are fixed, for all $t>0$, there exists an $N=N(\varepsilon,k,t)>0$ such that for all $n>N$, on $K$, we have
        \begin{equation*}
            |\tilde h(t)-\mathcal{D}(\lambda_n)\tilde g(t)|_{C^k(\tilde h_0)}\le\varepsilon.
        \end{equation*}
        On $K$, the curvature of $\mathcal{D}(\lambda_n)\tilde g(t)$ is uniformly bounded, that is, there exists a constant $C(K)>0$ such that on $K$, for all $\lambda_n>0, t\ge 0$ we have
        \begin{equation*}
            \sum _{i=0}^k|(\nabla^{\mathcal{D}(\lambda_n)\tilde g(t)})^i\Rm(\mathcal{D}(\lambda_n)\tilde g(t))|_{\mathcal{D}(\lambda_n)\tilde g(t)}\le C(K).
        \end{equation*}
        In particular, on $K$, there is a constant $C'>0$ such that for all $t\in[0,1]$, for all $n\in\N_0$
        \begin{equation*}
            |\mathcal{D}(\lambda_n)\tilde g(t)-\mathcal{D}'(\lambda_n)g_0|_{C^k(\tilde h_0)}\le C't.
        \end{equation*}
        Moreover, there exists $N'=N'(k,\varepsilon)>0$ such that for all $n>N'$, we have
        \begin{equation*}
            |\mathcal{D}'(\lambda_n)g_0-\tilde h_0|_{C^k(\tilde h_0)}\le\varepsilon,
        \end{equation*}
        holds on $K$. Combine the above results, we have on $K$, for $t\in [0,1]$,
        \begin{equation*}
            |\tilde h(t)-\tilde h_0|_{C^k(\tilde h_0)}\le 2\varepsilon+C't.
        \end{equation*}
        Then $\tilde h(t)$ converges to $\tilde h_0$ smoothly on $K$ as expected.
    \end{proof}
    \begin{claim}\label{claim 3}
    Let $\tilde h(t)_{t\in (0,\infty)}\in\varkappa(\omega_0)$ with $\mathcal{D}(\lambda_n)\tilde g(t)$ converging to $\tilde h(t)$ locally smooth on $M\times (0,\infty)$ for some positive sequence $\lambda_n\to 0$.
         For every $\varepsilon>0$ there exist $\delta(\varepsilon),T(\varepsilon)>0$ such that for all $t\in (0,T]$, we have
    \begin{equation*}
        \textnormal{Diam}_{\tilde h(t)}(\pi^{-1}\{r^2\le\delta\})\le \varepsilon.
    \end{equation*}
    \end{claim}
    \begin{proof}
        Recall that there exists a constant $C>1$ such that on $M\times [0,\infty)$, we have
        \begin{equation*}
            \frac{1}{C}g(t+1)\le \tilde g(t)\le Cg(t+1).
        \end{equation*}
        Therefore, on $M\times [0,\infty)$ we have
        \begin{equation*}
              \frac{1}{C}g(t+\lambda_n)\le \mathcal{D}(\lambda_n)\tilde g(t)\le Cg(t+\lambda_n).
        \end{equation*}
        For any $t>0$, the metrics $\mathcal{D}(\lambda_n)\tilde g(t)$ converges to $\tilde h(t)$ pointwisely, letting $\lambda_n\to 0^+$ tells us that
        \begin{equation*}
             \frac{1}{C}g(t)\le \tilde h(t)\le Cg(t).
        \end{equation*}
         Hence, for all $\delta>0$
        \begin{equation*}
             \textnormal{Diam}_{\tilde h(t)}(\pi^{-1}\{r^2\le\delta\})\le C\textnormal{Diam}_{g(t)}(\pi^{-1}\{r^2\le\delta\})=C\sqrt{t}\textnormal{Diam}_{g}(\pi^{-1}\{r^2\le\frac{\delta}{t}\}).
        \end{equation*}
        Recall that $\pi^{-1}\{r^2\le\frac{\delta}{t}\}\subset \{d_g(p,\cdot)\le \sqrt{\frac{\delta}{t}}+A\}$ holds for fixed point $p\in M$ and uniform constant $A>0$. As a consequence, we conclude that for all $\delta>0$, we have
        \begin{equation*}
             \textnormal{Diam}_{\tilde h(t)}(\pi^{-1}\{r^2\le\delta\})\le C\sqrt{t}(\sqrt{\frac{\delta}{t}}+A)=C(\sqrt{\delta}+A\sqrt{t}).
        \end{equation*}
        It suffices to take $T(\varepsilon)=\frac{\varepsilon^2}{4A^2C^2}$ and $\delta(\varepsilon)=\frac{\varepsilon^2}{4C^2}$.
    \end{proof}
      The proof of Theorem \ref{relation between two limit sets} follows directly from Claim \ref{claim 1}, \ref{claim 2} and \ref{claim 3}.

      In particular, if $\pi_*\tilde h_0$ is a K\"ahler conical metric, let us denote it by $g_{\mathcal{C}}'$, then $\tilde h(t)_{t\in (0,\infty)}$ is a solution to K\"ahler-Ricci flow with conical initial data $g_{\mathcal{C}}'$. Moreover, $JX$ always appears as a Killing vector field and there exists a uniform constant $C>0$ such that along the flow for all $t>0$
      \begin{equation*}
        \sup_M  |\Rm(\tilde h(t))|_{\tilde h(t)}\le \frac{C}{t}.
      \end{equation*}
      To use the uniqueness results in \cite{longteng1}, we only need to check the cohomological condition. Let $(M,g',X)$ be the unique gradient K\"ahler-Ricci expander with bounded curvature such that $\pi_*g'(t)$ converges to $g_{\mathcal{C}}'$ locally smoothly on $\mathcal{C}\setminus\{o\}$ thanks to the results in \cite{ConlonDeruelleJDG}. We only need to show there exists a real-valued smooth function $\varphi'$ on $M\times (0,\infty)$ such that for all $t>0$,
      \begin{equation*}
          \tilde h(t)=g'(t)+\partial\bar\partial \varphi'(t).
      \end{equation*}
      Since $g'$ and $g$ are in the same cohomology class, thus it suffices to show the existence of real-valued smooth function $\bar\varphi$ that is defined on $M\times (0,\infty)$ such that for all $t>0$,
      \begin{equation}\label{cohomology condition}
          \tilde h(t)=g(t)+\partial\bar\partial \bar\varphi(t).
      \end{equation}
      Recall that by definition $\tilde g(t)=g(t+1)+\partial\bar\partial \varphi(t)$. Now we show that $\mathcal{D}(\lambda_n)\varphi$ converges smoothly locally to some function $\bar\varphi$ then \eqref{cohomology condition} holds automatically. We already know there exists a uniform constant $C>0$ such that along the flow
      \begin{equation*}
          |\varphi(t)|\le C(t+1)f\circ\Phi_{t+1}+C(1+t),\quad |(\nabla^{g(t+1)})^k\partial\bar\partial\varphi(t)|_{g(t+1)}\le C_k((t+1)f\circ\Phi_{t+1})^{-k}.
      \end{equation*}
      Hence,
      \begin{equation*}
           |\mathcal{D}(\lambda_n)\varphi(t)|\le C(\lambda_n+t)f\circ\Phi_{\lambda_n+t}+C(\lambda_n+t),\quad |(\nabla^{g(t+\lambda_n)})^k\partial\bar\partial\mathcal{D}(\lambda_n)\varphi(t)|_{g(t+1)}\le C_k((\lambda_n+t)f\circ\Phi_{\lambda_n+t})^{-k}.
      \end{equation*}
      Then, up to subsequences, $\mathcal{D}(\lambda_n)\varphi(t)$ will converge to some function $\bar\varphi$ as required.
\end{proof}
  \begin{proof}[Proof of Corollary \ref{singleton}]
  If $\kappa(g_0)$ contains only a singleton, suppose that $\kappa(g_0)=\{g_{\mathcal{C}}'\}$. We prove $g_{\mathcal{C}}'$ is a conical metric on $\mathcal{C}$. Let $\omega_{\mathcal{C}}'$ be its K\"ahler form. Since for all $\lambda,\mu>0$ we have
  \begin{equation*}
      \mathcal{D}'(\lambda)D'(\mu)\omega_0=\mathcal D'(\lambda\mu)\omega_0,
  \end{equation*}
  we conclude that $\mathcal{D}'(\lambda)\omega_{\mathcal{C}}'=\omega_{\mathcal{C}}'$ for all $\lambda>0$. Therefore,
  \begin{equation*}
      \mathcal{L}_X\omega_\mathcal{C}'=2\omega_\mathcal{C}'.
  \end{equation*}Moreover, $\omega_0=\omega+i\partial\bar\partial \psi_0$, it follows that there exists some smooth function $\psi'$ on $\mathcal{C}\setminus \{o\}$ such that
  \begin{equation*}
      \omega_{\mathcal{C}}'=\omega_\mathcal{C}+i\partial\bar\partial \psi'.
  \end{equation*}
  Proposition \ref{rough estimates} implies that on $\mathcal{C}\setminus \{o\}$, we have
  \begin{equation*}
      |\psi'|\le (e^a-1)\frac{r^2}{2},\quad |X\cdot\psi'|\le (e^a-1)r^2.
  \end{equation*}
  We compute
  \begin{equation*}
       \mathcal{L}_X\omega_\mathcal{C}+i\partial\bar\partial(X\cdot\psi') =\mathcal{L}_X\omega_\mathcal{C}'=2\omega_\mathcal{C}'+2i\partial\bar\partial\psi'.
  \end{equation*}
  It follows that
  \begin{equation*}
      i\partial\bar\partial(X\cdot\psi'-2\psi')=0.
  \end{equation*}
  Recall that by the $JX-$invariance, we also have
  \begin{equation*}
      JX\cdot\psi'=0.
  \end{equation*}
  Hence, the function $X\cdot\psi'-2\psi'$ is a $JX-$invariant pluri-harmonic function on $\mathcal{C}\setminus \{o\}$. By the Liouville-type lemma stated in \cite{longteng1}, there exist $c,b\in\R$ such that
  \begin{equation*}
      X\cdot\psi'-2\psi'=c\log r+b.
  \end{equation*}
  Moreover the previous bounds show that, in particular, function $X\cdot\psi'-2\psi'$ tends to 0 around $o$, then we have $c=b=0$, and therefore $\omega_\mathcal{C}'$ is K\"ahler conical metric on the cone $\mathcal{C}$. When in this case, by Theorem \ref{relation between two limit sets}, $\varkappa(g_0)$ contains only the self-similar solution $g'(t)_{t>0}$, such that $(M,g'(1),X)$ is the unique gradient K\"ahler-Ricci expander with bounded curvature such that the corresponding self-similar solution $\pi_*\omega'(t)_{t>0}$ converges locally smoothly to $\omega_\mathcal{C}'$ on $\mathcal{C}\setminus\{o\}$ as $t$ tends to 0.

  Conversely, we suppose that $\tilde g(t)_{t\ge0}$ in Corollary \ref{singleton} is asymptotically self-similar, i.e, $\varkappa(g_0)=\{g'(t)_{t>0}\}$. Then similarly, for all $\lambda>0$, we have

  \begin{equation*}
      \mathcal{D}(\lambda)g'(t)=g'(t).
  \end{equation*}
  In particular, we have for any $\lambda>0$, it holds that
  \begin{equation*}
      g'(\lambda)=\mathcal{D}(\lambda)g'(\lambda)=\lambda\Phi_\lambda^*g'(1).
  \end{equation*}
  Then $(M,g'(1),X)$ is an expanding K\"ahler-Ricci soliton. Let $g':=g'(1)$, since $X$ is gradient along the flow $\tilde \omega(t)_{t\ge 0}$, then $X$ s gradient with respect to $g'$. The metric $g'$ has quadratic curvature decay as shown in \ref{Theorem B i} of Theorem \ref{relation between two limit sets}, then $g'(t)$ converges to a K\"ahler conical metric $g_\mathcal{C}'$ on $\mathcal{C}\setminus\{o\}$.
  \end{proof}
\begin{proof}[Proof of Corollary \ref{singular data}]
    Let $g_{\textnormal{sing}}$ be defined as in Corollary \ref{singular data}. For any $s>0$ we define the initial metric on $M$ as follows:
    \begin{equation*}
        \omega_0^s:=\omega+i\partial\bar\partial\left(\chi(r(\cdot)N^{-1})\varphi_{\textnormal{sing}
        }(s)\right),
    \end{equation*}
    with constant $N>0$ to be determined.
    Here $\varphi_{\textnormal{sing}}(s)$ is defined as $\varphi_{\textnormal{sing}}(s)=\frac{1}{s}\Phi_{\frac{1}{s}}^*\varphi_{\textnormal{sing}}$. Similar to \cite[Proposition 3.3]{ChenHallgrenLucas}, one can show that if $N^{-1}$ and $a'$ in Corollary \ref{singular data} is sufficiently small, then $\omega_0^s$ defines a metric. Indeed, on $\{4N^2\ge r^2\ge N^2\}$ we compute for each $k\in\N_0$,
    \begin{equation*}
       \begin{split}
            r^k|(\nabla^{g_\mathcal{C}})^k(g_0^s-g_\mathcal{C})|_{g_C}&\le |(\nabla^{g_\mathcal{C}})^k(g-g_\mathcal{C})|_{g_\mathcal C}+C_k\sum_{i=0}^{k+2}N^{i-2-k}|(\nabla^{g_\mathcal{C}})^{i}\varphi_{\operatorname{sing}}(s)|_{g_\mathcal{C}}\\
            &\le C_kN^{-2}+C_k.
       \end{split}
    \end{equation*}   
In particular,
\begin{equation*}
    |g_0^s-g_\mathcal{C}|_{g_C}\le |g-g_\mathcal{C}|_{g_\mathcal C}+C\sum_{i=0}^{2}N^{i-2}|(\nabla^{g_\mathcal{C}})^{i}\varphi_{\operatorname{sing}}(s)|_{g_\mathcal{C}}\le CN^{-2}+Ca'.
\end{equation*}
        Moreover, we also have that
        \begin{equation*}
       \operatorname{BL}_{\infty}(g_0^s,g)=     \operatorname{BL}_{\infty}(g_0^s,g_{\mathcal C})\le |(\nabla^{g_\mathcal{C}})^2\varphi_{\operatorname{sing}}(s)|_{g_\mathcal{C}}\le a'.
        \end{equation*}
    If we take $a'\le a$ as in Theorem \ref{main estimate thm}, then let $\omega_s(t)_{t\ge0}$ be the unique Shi's K\"ahler-Ricci flow with initial metric $\omega_0^s$, the results in Theorem \ref{main estimate thm} hold for $\omega_s(t)_{t\ge0}$ for all $s>0$. Notice that the constants $C,\{C_k\}_{k\in\N_0}$ in Theorem \ref{main estimate thm} only depend on the initial value, now we consider $\omega^s(t):=s\Phi_s^*\omega_s(t/s)$, then it holds that for all $s>0$
  \begin{enumerate}
      \item There exists a constant $C>1$ such that
      \begin{equation*}
          \frac{1}{C}\omega(t+s)\le \omega^s(t)\le C\omega(t+s)
      \end{equation*}
      \item For $p\in E$, for all $k\in\N_0$, there is a constant $C_k>0$ such that
      \begin{equation*}
         \sup_M (\sqrt{t+s}+d_{g(t+s)}(p,\cdot))^k|(\nabla^{g(t+s)})^kg^s(t)|_{g(t+s)}\le C_k.
      \end{equation*}
  \end{enumerate}
  We can find a subsequence $\{s_k\}_{k\in\N_0}$ tending to 0 such that $\omega^s(t)$ converges locally smoothly to a solution of K\"ahler-Ricci flow $\omega_{\textnormal{smooth}}(t)_{t>0}$. Moreover, we observe that $\omega^s(0)$ converges locally smoothly to $\omega_{\textnormal{sing}}$ as $s$ tends to 0. By the same argument in the proof of Theorem \ref{relation between two limit sets}, we claim that \ref{Theorem B i}-\ref{Theorem B iv} in Theorem \ref{relation between two limit sets} hold for $\omega_{\operatorname{smooth}}(t)_{t\in (0,\infty)}$ on $M$ and $\pi^*\omega_{\operatorname{sing}}$ on $M\setminus E$.
\end{proof}
\section{Universal initial data}\label{universal initial values}
In this section, we show the existence of a K\"ahler-Ricci flow that can exhibit genuinely different asymptotic behaviors along different sequences of times tending to infinity, as well as under different choices of time-dependent rescalings. To do this, we prove Theorem \ref{universal initial data}
\begin{proof}[Proof of Theorem \ref{universal initial data}]
    Let $\delta>0$ to be determined later, we consider the metric ball $\mathcal{B}(g_{\mathcal{C}},1)$. Every K\"ahler conical metric $g_\mathcal{C}'$ can be related to a $\xi-$invariant smooth function $\psi_S$ that is defined on the link.

    \begin{claim}\label{claim 1 thm E}
        There exists an injection from $\mathcal{B}(g_{\mathcal{C}},1)$ to $C_{\xi}^\infty(S;\R)$:
        \begin{equation*}
            \digamma: \mathcal{B}(g_{\mathcal{C}},1)\mapsto C_{\xi}^\infty(S;\R).
        \end{equation*}
        Here $C_{\xi}^\infty(S;\R)$ denote the set of smooth real-valued functions that are $\xi$-invariant
    \end{claim}
  \begin{proof}  To see this, let $R$ denote the radial function with respect to $g_{\mathcal{C}}'$ such that 
    \begin{equation*}
        \omega_\mathcal{C}'=i\partial\bar\partial(\frac{R^2}{2});\quad g_{\mathcal{C}}'=dR^2+R^2g_S'.
    \end{equation*}
    Here $g_S'$ is a Riemannian metric that is defined on the link $S$. Since $g_{\mathcal{C}}'$ is $\xi-$invariant, it follows that the radial vector field $R\partial_R=X$. Then we conclude that
    \begin{equation*}
        \mathcal{L}_X\omega_\mathcal{C}'=2\omega_\mathcal{C}',
    \end{equation*}
    therefore we have
    \begin{equation*}
        i\partial\bar\partial(rR\partial_r R-R^2)=0.
    \end{equation*}
    Recall that we also have $\xi\cdot(rR\partial_r R-R^2)=0$, then by the Liouville-type lemma stated in \cite{longteng1}, we conclude that 
    \begin{equation*}
        rR\partial_r R=R^2,
    \end{equation*}
    it yields that $\frac{R^2}{r^2}$ is a function that is only defined on the link $S$. Hence we have obtained an injection between $\mathcal{B}(g_{\mathcal{C}},1)$ and a space of functions on $S$. We denote this map as $\digamma$, more precisely, for any $g_{\mathcal{C}}'\in \mathcal{B}(g_{\mathcal{C}},1)$,
    \begin{equation*}
        \digamma(g_{\mathcal{C}}')=\psi',
    \end{equation*}
    where $\psi'$ is the unique smooth function on $S$ such that $\omega_{\mathcal{C}}'=i\partial\bar\partial\left((1+\psi')\frac{r^2}{2}\right)$.
\end{proof}    
    For all $\psi_1,\psi_2\in C_{\xi}^\infty(S;\R)$, define
    \begin{equation*}
        |\psi_1-\psi_2|_{\operatorname{link}}=|\psi_1-\psi_2|_{C^1(S)}+|i\partial\bar\partial (\psi_1-\psi_2)|_{g_S}
    \end{equation*}
    We notice that $\digamma(g_\mathcal{C})=0$, we view $\digamma(\mathcal{B}(g_{\mathcal{C}},1))$ as a subset of $(C_{\xi}^\infty(S;\R),|\cdot|_{\operatorname{link}})$. 
\begin{claim}
    There exists an $\varepsilon_0>0$ such that $\digamma$ is an onto on the metric ball $B(0,\varepsilon_0)\subset (C_{\xi}^\infty(S;\R),|\cdot|_{\operatorname{link}})$.
\end{claim}
\begin{proof}
    Notice that there exists a dimensional constant $C(n)>0$ such that for any $\psi'\in C_{\xi}^\infty(S;\R)$, we have
    \begin{equation*}
        |i\partial\bar\partial(\psi'\frac{r^2}{2})|_{g_\mathcal C}\le C(n)|\psi'|_{\operatorname{link}}.
    \end{equation*}
    We can take $\varepsilon_0$ small enough such that $i\partial\bar\partial(\psi'\frac{r^2}{2})+\omega_\mathcal{C}$ defines a metric that is close to $\omega_\mathcal{C}$.
\end{proof}
\begin{claim}
    There exists an $\delta_0>0$ small enough such that $\digamma(\mathcal{B}(g_{\mathcal{C}},\delta_0))\subset B(0,\varepsilon_0)$.
\end{claim}
\begin{proof}
To see this, we show that for any $g_{\mathcal{C}}'\in \mathcal{B}(g_{\mathcal{C}},\delta_0)$, when $\delta_0$ is sufficiently small,
    \begin{equation*}
    |\digamma(g_{\mathcal{C}}')|_{\operatorname{link}}\le  \varepsilon_0.
    \end{equation*}
    Suppose $\digamma(g_{\mathcal{C}}')=\psi'$, then we have
\begin{equation*}
   e^{-\delta_0}\omega_\mathcal C\le \omega_{\mathcal{C}}'=i\partial\bar\partial\left((\psi'+1)\frac{r^2}{2}\right)\le e^{\delta_0}\omega_\mathcal C.
\end{equation*}
It follows that
\begin{equation*}
   e^{-\delta_0}\frac{r^2}{2}\le \frac{1}{2} i\partial\bar\partial\left((\psi'+1)\frac{r^2}{2}\right)(r\partial_r,\xi)=(1+\psi')\frac{r^2}{2}\le e^{\delta_0}\frac{r^2}{2}.
\end{equation*}
Therefore, we get $|\psi'|_{C^0(S)}\le e^{\delta_0}-1$. Moreover, $|i\partial\bar\partial \psi'|_{g_S}\le e^{\delta_0}-1$ is obvious by definition.

Now we compute, for all $Y$ a vector field on $S$, 
\begin{equation*}
 |Y\cdot(X\cdot(\psi' r^2))|\le 4(e^{\delta_0}-1)r|Y|_{g_S}|X|_{g_\mathcal{C}}.
\end{equation*}
We get
\begin{equation*}
    |Y\cdot\psi'|\le 2(e^{\delta_0}-1)|Y|_{g_S}.
\end{equation*}
And therefore, $|d\psi'|_{g_S}\le 2(e^{\delta_0}-1)$. This completes the proof.
\end{proof}
  Now we take $\delta>0$ sufficiently small such that $\digamma(\mathcal{B}(g_{\mathcal{C}},\delta))\subset B(0,\varepsilon)$ with $\varepsilon>0$ to be determined, now we consider the inverse map
    \begin{equation*}
        \digamma^{-1}: B(0,\varepsilon)\mapsto \mathcal{B}(g_{\mathcal{C}},1)),
    \end{equation*}
    then $\mathcal{B}(g_{\mathcal{C}},\delta)\subset \digamma^{-1}(B(0,\varepsilon))$ and 
    there exists a bi-Lipschitz constant $L>0$ such that for any $\psi_1,\psi_2\in B(0,\varepsilon)$,
    \begin{equation*}
        \operatorname{BL}(\digamma^{-1}(\psi_1),\digamma^{-1}(\psi_2))\le L|\psi_1-\psi_2|_{C_{g_S}^2}.
    \end{equation*}
    Let $\{\psi_k\}_{k\in\N}$ be a countable dense sequence in $B(0,\varepsilon)$, hence $\mathcal{B}(g_{\mathcal{C}},\delta)\subset\overline{\{\digamma^{-1}(\psi_k)\}_{k\in\N}}$.

    Let $a_n=(2n+2)!$ and $b_n=(2n+3)!$. Using a bijection $\iota:\N\to\N^2$, we rename the sequences as $\{a_{k,l}=a_{\iota^{-1}(k,l)}\}_{k\in\N,l\in\N}$ and $\{b_{m,n}=b_{\iota^{-1}(m,n)}\}_{m\in\N,n\in\N}$. It follows that if $(k,l)\neq (m,n)$, then 
    \begin{equation*}
        (a_{k,l},b_{k,l})\cap (a_{m,n},b_{m,n})=\emptyset.
    \end{equation*}
    We now set $\lambda_{k,l}=\sqrt{a_{k,l}b_{k,l}}$, then we have for any $k\in\N$,
    \begin{equation*}
        \lim_{l\to\infty}\frac{a_{k,l}}{\lambda_{k,l}}=0,\quad \lim_{l\to \infty}\frac{b_{k,l}}{\lambda_{k,l}}=\infty.
    \end{equation*}
    Let $1\ge \chi_{k,l}\ge 0$ be a cut-off function such that 
    \begin{equation*}
        (a_{k,l}\le r^2\le b_{k,l})\subset \supp \chi_{k,l}\subset (c_{k,l},d_{k,l}),
    \end{equation*}
    where
    \begin{equation*}
       c_{k,l}=\frac{a_{k,l}+b_{\iota(\iota^{-1}(k,l)-1)}}{2},\quad d_{k,l}=\frac{b_{k,l}+a_{\iota(\iota^{-1}(k,l)+1)}}{2}.
    \end{equation*}
Then it follows that $\supp\chi_{k,l}\cap\supp\chi_{m,n}=\emptyset$ if $(k,l)\neq(m,n)$. We can also take such cut-off functions such that for all $k,l\in\N$
    \begin{equation*}
       \begin{split}
           &\sup_\R |\chi_{k,l}'|(|d_{k,l}-b_{k,l}|+|a_{k,l}-c_{k,l}|)\le 100;\\
           & \sup_\R|\chi_{k,l}''|(|d_{k,l}-b_{k,l}|^2+|a_{k,l}-c_{k,l}|^2)\le 100.
       \end{split}
    \end{equation*}

    Now we define 
    \begin{equation*}
        \omega_0=\omega+i\partial\bar\partial\left(\sum_{k\in\N}\sum_{l\in\N}\chi_{k,l}(r(\cdot))\frac{r^2}{2}\psi_k\right).
    \end{equation*}
    First we show that $\omega_0$ is a metric that is $a-$close to $\omega$. We compute
\begin{equation*}
    \begin{split}
        |\omega_0-\omega|_{g_\mathcal{C}}&\le C(n)\sum_{k,l\in\N}|\psi_k|_{C^2(S)}\left(\chi_{k,l}(r(\cdot))+r|\nabla^{g_\mathcal{C}}\chi_{k,l}(r(\cdot))|_{g_{\mathcal{C}}}+r^2|(\nabla^{g_\mathcal{C}})^2\chi_{k,l}(r(\cdot))|_{g_{\mathcal{C}}}\right)\\
        &\le \varepsilon C(n)(1+\frac{100d_{k,l}}{|d_{k,l}-b_{k,l}|}+\frac{100d_{k,l}^2}{|d_{k,l}-b_{k,l}|^2}).
    \end{split}
\end{equation*}
Let $\iota(q)=(k,l)$, then we have
\begin{equation*}
    \frac{d_{k,l}}{|d_{k,l}-b_{k,l}|}+\frac{d_{k,l}^2}{|d_{k,l}-b_{k,l}|^2}=\frac{a_{q+1}+b_q}{a_{q+1}-b_q}+\left(\frac{a_{q+1}+b_q}{a_{q+1}-b_q}\right)^2\le \frac{2q+5}{2q+3}+(\frac{2q+5}{2q+3})^2<5.
\end{equation*}
Then we have
\begin{equation*}
    |\omega_0-\omega|_{g_\mathcal{C}}\le \varepsilon 501C(n).
\end{equation*}
Now take $\varepsilon$ small enough such that $\omega_0$ is $a-$close to $\omega$, and we investigate $\kappa(\omega_0)$. Since $\lim_{n\to\infty}\lambda_{m,n}=\infty$, we consider
\begin{equation*}
\mathcal{D}'(\lambda_{m,n}^{-2})\omega_0=\omega(\lambda_{m,n}^{-2})+i\partial\bar\partial\left(\sum_{k\in\N}\sum_{l\in\N}\chi_{k,l}(r(\cdot)\lambda_{m,n})\frac{r^2}{2}\psi_k\right).
\end{equation*}
By the definition of cut-off functions, on the annulus $\{\frac{a_{m,n}}{\lambda_{m,n}}\le r\le \frac{b_{m,n}}{\lambda_{m,n}}\}$, we have
\begin{equation*}
    \mathcal{D}'(\lambda_{m,n}^{-2})\omega_0=\omega(\lambda_{m,n}^{-2})+i\partial\bar\partial \left(\frac{r^2}{2}\psi_m\right).
\end{equation*}
Since $\lim_{n\to\infty}\frac{a_{m,n}}{\lambda_{m,n}}=0$ and $\lim_{n\to\infty}\frac{b_{m,n}}{\lambda_{m,n}}=\infty$, we conclude that for any $m\in\N$,
\begin{equation*}
    i\partial\bar\partial\left(\frac{r^2}{2}(1+\psi_m)\right)\in \kappa(\omega_0).
\end{equation*}
Hence $\kappa(\omega_0)$ is dense in $\mathcal{B}(g_{\mathcal{C}},\delta)$ with respect to the metric $\operatorname{BL}$.
\end{proof}
\section{Dynamical system of K\"ahler-Ricci flows on the gradient AC expanders}\label{section dyn system}
Another way of understanding the long-time
asymptotic behavior of the solutions to K\"ahler-Ricci flow is to study the dynamical system generated by the flows. In this section, we study the dynamical behavior of K\"ahler--Ricci flows on the AC gradient K\"ahler-Ricci expander $(M,g,X)$. We construct the associated topological dynamical system and show that it exhibits chaotic behavior.

Throughout this section, we let $(M,g,X)$(resp. $\mathcal{C},g_{\mathcal{C}}$) denote the AC gradient K\"ahler-Ricci expander(resp. K\"ahler cone). Let $\omega$ be the K\"ahler form associated with $g$ and let $f$ be the soliton potential. Let $\pi: M\mapsto\mathcal{C}$ be the K\"ahler resolution as in Theorem \ref{theorem of conlonderuellesun}. We identify $M\setminus E$ with $\mathcal{C}\setminus\{o\}$ via this biholomorphism. Let $\Phi_t$ be the flow of $-\frac{X}{2t}$ for $t>0$. Then the self-similar solution associated with $(M,g,X)$ is given by $g(t):=t\Phi_t^*g.$

\subsection{Dynamical system of bounded solutions to K\"ahler-Ricci flow}
\begin{definition}[$(C,B)-$Bounded solutions K\"ahler-Ricci flow]
Let $C>1$ and $B=\{B_k\}_{k\in\N^*}$, a $(C,B)-$\emph{bounded solution to K\"ahler-Ricci flow} is a K\"ahler-Ricci flow $(g_{\operatorname{bd}}(t))_{t>0}$ such that the following hold:
    \begin{enumerate}
        \item There exists a smooth, real-valued function $\varphi\in C^\infty(M\times (0,\infty))$ such that
        \begin{equation*}
            \omega_{\operatorname{bd}}(t)=\omega(t)+i\partial\bar\partial \varphi(t).
        \end{equation*}
        Here $\omega_{\operatorname{bd}}(t)$ is the K\"ahler form of $g_{\operatorname{bd}}(t)$ for all $t>0$.
        \item On $M\times (0,\infty)$, for all $k\in\N^*$, we have
        \begin{equation*}
            \begin{split}
                &C^{-1}g(t)\le  g_{\operatorname{bd}}(t)\le C g(t).\\
                & \sup_{M}(t\Phi_t^*f)^{k/2}|(\nabla^{g(t)})^k g_{\operatorname{bd}}(t)|_{g(t)}\le B_k.
            \end{split}
        \end{equation*}
        \item The flow $(g_{\operatorname{bd}}(t))_{t>0}$ is $JX$-invariant.
    \end{enumerate}
    Let $\mathcal{N}_{C}^B$ denote the set of all $(C,B)$-bounded solutions to K\"ahler-Ricci flows.
\end{definition}
Now we give $\mathcal{N}_{C}^B$ a suitable topology.
 For all $0<t_0<t_1<\infty,L>0$, let $\mathscr{R}_{[t_0,t_1]}^L$ denote the space-time rectangle $\{r^2\le L\}\times [t_0,t_1]$.
\begin{definition}[Topology]\label{Topology}
    We give a topology on $\mathcal{N}_{C}^B$. For all $(\omega_1(t))_{t>0}\in \mathcal{N}_{C}^B,\varepsilon>0,0<t_0<t_1<\infty,L>0$, define
    \begin{equation*}
        U^\varepsilon(\omega_1(t),{\mathscr{R}_{[t_0,t_1]}^L}):=\{(\omega_2(t))_{t>0}\in\mathcal{N}_C^B\ |\ |\omega_2(t)-\omega_1(t)|_{g(t)}<\epsilon \textnormal{ on $\mathscr{R}_{[t_0,t_1]}^L$}\}.
    \end{equation*}
    Let $\mathcal{T}$ denote the topology generated by $\{ U^\varepsilon(\omega_1,{\mathscr{R}})\}_{\varepsilon>0,\omega_1\in \mathcal{N}_{C}^B}$.  
\end{definition}
\begin{definition}[Metric structure]
    For all $(\omega_1(t))_{t>0},(\omega_2(t))_{t>0}\in \mathcal{N}_C^B$, define
    \begin{equation*}
        d(\omega_1,\omega_2)=\sum_{i,j\in\N_0}e^{-i-j}\sup_{\mathscr{R}_{[e^{-i},e^i]}^{e^j}}|\omega_1(t)-\omega_2(t)|_{g(t)}.
    \end{equation*}
\end{definition}
\begin{prop}
     The ordered pair $(\mathcal{N}_C^B,d)$ is a metric space. And the topology induced by metric $d$ is equivalent to $\mathcal{T}$.
\end{prop}
\begin{proof}
    Obviously, $d$ defines a metric.

    Now we prove the equivalent of two topologies. It is obvious that the metric topology of $(\mathcal{N}_C^B,d)$ is finer that $\mathcal{T}$. Notice that there is a constant $C>0$ such that
    \begin{equation*}
      |\omega_1(t)-\omega(t)|_{g(t)}\le C,
    \end{equation*}
    holds for all $\omega_1(t)\in \mathcal{N}^C_B$ as in Corollary \ref{coro: definition of F}. Then for all $\varepsilon>0,\omega_1(t)\in \mathcal{N}_C^B$, consider $\omega_2(t)\in U^\delta(\omega_1,\mathscr{
    R}_{[e^{-k},e^k]}^{e^j})$ with $k,j\in\N^*,\delta>0$ to be determined. We compute
    \begin{equation*}
      \begin{split}
          d(\omega_1,\omega_2)&\le \sum_{\ell\le k,p\le j}e^{-\ell-p}\delta+\sum_{\ell>k}e^{-\ell}|\omega_1(t)-\omega_2(t)|_{g(t)}+\sum_{p>j}e^{-p}|\omega_1(t)-\omega_2(t)|_{g(t)}\\
          &\le \sum_{\ell\le k,p\le j}e^{-\ell-p}\delta+2C\sum_{\ell>k}e^{-\ell}+2C\sum_{p>j}e^{-j}\\
          &\le \delta \frac{1}{(1-e^{-1})^2}+2Ce^{-k-1}\frac{1}{1-e^{-1}}+2Ce^{-j-1}\frac{1}{1-e^{-1}}.
      \end{split}
    \end{equation*}
    Take $\delta<<1$ and $k,j>>1$, then we have $ d(\omega_1,\omega_2)<\varepsilon$. It follows immediately that two topologies are equivalent.
\end{proof}
For all $(\omega_{\operatorname{smooth}}(t))_{t>0}\in \mathcal{N}_C^B$, for all $\tau\in\R$, we define
\begin{equation*}
    \phi^\tau(\omega_{\operatorname{smooth}}(t))=\mathcal{D}(e^{-\tau})\omega_{\operatorname{smooth}}(t)=e^{-\tau}\Phi_{e^{-\tau}}^*\omega_{\operatorname{smooth}}(e^{\tau}t).
\end{equation*}
\begin{prop}[Topological dynamical system of K\"ahler-Ricci flows]
    The metric space $(\mathcal{N}_C^B,d)$ together with the map $(\phi^\tau)_{\tau\in \R}$ generate a topological dynamical system $(\mathcal{N}_C^B,d,(\phi^\tau)_{\tau\ge 0})$.
\end{prop}
\begin{proof}
    First, we prove that $\phi^\tau(\mathcal{N}_C^B)\subset \mathcal{N}_C^B$ for all $\tau\ge 0$. Let $(\omega_1(t))_{t>0}$ be a $(C,B)$-bounded solution to K\"ahler-Ricci flow. Then it follows that $\phi^\tau(\omega_1(t))=(e^{-\tau}\Phi_{e^{-\tau}}^*\omega_1(e^\tau t))_{t>0}$ is a $(C,B)$-bounded solution to K\"ahler-Ricci flow. Still, by the definition of $\phi^\tau$, we can see that $\phi^\tau\circ \phi^{\tau'}=\phi^{\tau+\tau'}$ and $F^0=\operatorname{id}$.

    Now we prove that the map: $\phi: \R\times \mathcal{N}_C^B\mapsto\mathcal{N}_C^B$ is continuous.

    It suffices to show, for all $\varepsilon>0,\mathscr{R}_{[t_0,t_1]}^L,(\omega_1(t))_{t>0}\in\mathcal{N}_C^B,\tau_0\ge 0$, there exist $\delta,L'>0,0<t_0'<t_1'<\infty$, a neighborhood $\mathcal{I}(\tau_0)$ of $\tau$ such that for all $(\omega_1'(t))_{t>0}\in U^\delta(\omega_1,\mathscr{R}_{[t_0',t_1']}^{L'}),\tau\in\mathcal{I}(\tau_0)$, we have $\phi^{\tau}(\omega_1')\in U^\varepsilon(\phi^{\tau_0}(\omega_1),\mathscr{R}_{[t_0,t_1]}^L)$. Observe that, if 
    \begin{equation*}
        |\omega_1(t)-\omega_1'(t)|_{g(t)}\le \delta,
    \end{equation*}
holds in $\mathscr{R}_{[t_0',t_1']}^{L'}$, then for all $\tau\in\R$
\begin{equation*}
    |(\phi^\tau\omega_1)(t)-(\phi^\tau\omega_1')(t)|_{g(t)}\le \delta,
\end{equation*}
holds in $\mathscr{R}_{[e^{-\tau}t_0',e^{-\tau}t_1']}^{L'e^{-\tau}}$. Take $t_0'\le e^{\tau_0}t_0, t_1'\ge e^{\tau_0}t_1, L'\ge Le^{\tau_0}$, then we get
\begin{equation*}
     |(\phi^{\tau_0}\omega_1)(t)-(\phi^{\tau_0}\omega_1')(t)|_{g(t)}\le \delta,
\end{equation*}
holds in $\mathscr{R}_{[t_0,t_1]}^{L}$. Moreover, for all $\tau\in \R$, we have
\begin{equation*}
    (\phi^{\tau}\omega_1')(t)-(\phi^{\tau_0}\omega_1')(t)=\int_{\tau_0}^\tau\left(\mathcal{L}_{\frac{X}{2}}(\phi^\rho
\omega_1')(t)-(\phi^\rho
\omega_1')(t)-t\Ric((\phi^\rho
\omega_1')(t))\right)d\rho\end{equation*}
Notice that $\left|\left(\mathcal{L}_{\frac{X}{2}}(\phi^\rho
\omega_1')(t)-(\phi^\rho
\omega_1')(t)-t\Ric((\phi^\rho
\omega_1')(t))\right)\right|_{g(t)}\le C$ holds uniformly for all $t>0, \omega_1'\in\mathcal{N}^C_B$, hence
\begin{equation*}
    |(\phi^{\tau}\omega_1')(t)-(\phi^{\tau_0}\omega_1')(t)|_{g(t)}\le C|\tau_0-\tau|.
\end{equation*}
Then the continuity on $\tau$ follows as well.

\end{proof}
\subsection{Dynamical system of K\"ahler metrics on K\"ahler cone}
\begin{definition}\label{the set Kappa}
    For $\sigma>0$, $b=\{b_k>0\}_{k\in\N^*}$, we say a K\"ahler metric $\omega_{\operatorname{sing}}$ defined on $\mathcal{C}\setminus\{o\}$ is an element of $\mathcal{K}_\sigma^b$ if the following hold:
\begin{enumerate}
    \item There is a $JX-$invariant smooth function $\varphi_{\operatorname{sing}}$ defined on $\mathcal{C}\setminus\{o\}$ such that
    \begin{equation*}
        \omega_{\operatorname{sing}}=\omega_{\mathcal{C}}+i\partial\bar\partial \varphi_{\operatorname{sing}}.
    \end{equation*}
    \item For all $r>0$, we have
    \begin{equation*}
         \sup_{\mathcal{C}\setminus\{o\}} \max\{r^{-2}|\varphi_{\operatorname{sing}}|,r^{-1}|\nabla^{g_\mathcal{C}}\varphi_{\operatorname{sing}}|_{g_{\mathcal{C}}},|(\nabla^{g_\mathcal{C}})^2\varphi_{\operatorname{sing}}|_{g_\mathcal{C}}\}\le \sigma.
    \end{equation*}
    \item For each $k\in\N^*$, we have
    \begin{equation*}
        \sup_{\mathcal{C}\setminus\{o\}} r^k|(\nabla^{g_\mathcal{C}})^{k+2}\varphi_{\operatorname{sing}}|_{g_\mathcal{C}}\le b_k
    \end{equation*}
\end{enumerate}
\end{definition}
For all $0<r_0<R_0<\infty$, define $\mathscr{A}_{r_0,R_0}:=\{r_0^2\le r^2\le R_0^2\}$ on $\mathcal{C}$. In particular, for $k\in\N^*$, let $\mathscr{A}_k$ denote $\mathscr{A}_{e^{-k},e^k}$.
Similarly, for each $\tau\in\R$, we can also define a topological dynamical system $(\mathcal{K}^b_{\sigma},d_{\mathcal{C}},(\phi^\tau_{\mathcal{C}})_{\tau\in\R})$ as follows:

\begin{prop}[Topological dynamical system of K\"ahler metrics]
    Let $\mathcal{K}^b_\sigma$ be the set of K\"ahler metrics as in Definition \ref{the set Kappa}. For $\omega_{\operatorname{sing}}^1,\omega_{\operatorname{sing}}^2\in \mathcal{K}^b_\sigma$, define 
    \begin{equation*}
        d_{\mathcal{C}}(\omega_{\operatorname{sing}}^1,\omega_{\operatorname{sing}}^2):=\sum_{k\in\N^*}e^{-k}\sup_{\mathscr{A}_k}|\omega_{\operatorname{sing}}^1-\omega_{\operatorname{sing}}^2|_{g_\mathcal{C}}.
    \end{equation*}
    For $\tau\in\R,\omega_{\operatorname{sing}}\in \mathcal{K}^b_\sigma$, we also define 
    \begin{equation*}
        \phi_\mathcal{C}^\tau(\omega_{\operatorname{sing}})=e^{-\tau}\Phi_{e^{-\tau}}^*\omega_{\operatorname{sing}}.
    \end{equation*}
    Then $(\mathcal{K}^b_{\sigma},d_{\mathcal{C}},(\phi^\tau_{\mathcal{C}})_{\tau\in\R})$ is a topological dynamical system.
\end{prop}
\begin{proof}
    It is obvious that $d_{\mathcal{C}}$ defines a metric and $(\phi_{\mathcal{C}}^\tau)_{\tau\in\R}$ is an additive group. 
    
    Notice that if \begin{equation*}
        |\omega_{\operatorname{sing}}^1-\omega_{\operatorname{sing}}^2|_{g_\mathcal{C}}\le \delta,
    \end{equation*}
    holds on $\mathscr{A}_{r_0,R_0}$, then 
    \begin{equation*}
       | \phi^\tau_\mathcal{C}(\omega_{\operatorname{sing}}^1)-\phi^\tau_\mathcal{C}(\omega_{\operatorname{sing}}^2)|_{g_{\mathcal
       C}}\le \delta,
    \end{equation*}
    holds on $\mathscr{A}_{e^{-\tau/2}r_0,e^{-\tau/2}R_0}$. Hence, for all $\omega_{\operatorname{sing}}^1\in\mathcal{K}^b_\sigma,\tau_0\in\R$, for all $0<r_1<R_1<\infty$, $\varepsilon>0$, we can take $\omega_{\operatorname{sing}}^2\in\mathcal{K}^{b}_\sigma$ such that 
    \begin{equation*}
        |\omega_{\operatorname{sing}}^2-\omega_{\operatorname{sing}}^1|_{g_{\mathcal{C}}}\le\varepsilon/2,
    \end{equation*}
    holds on $\mathscr{A}_{e^{\tau_0/2}r_1,e^{\tau_0/2}R_1}$. Then it follows that
    \begin{equation*}
          | \phi^{\tau_0}_\mathcal{C}(\omega_{\operatorname{sing}}^1)-\phi^{\tau_0}_\mathcal{C}(\omega_{\operatorname{sing}}^2)|_{g_{\mathcal
       C}}\le \varepsilon/2,
    \end{equation*}
     holds on $\mathscr{A}_{r_1,R_1}$. Moreover, for all $\tau\in\R$, we have
     \begin{equation*}
           | \phi^{\tau}_\mathcal{C}(\omega_{\operatorname{sing}}^2)-\phi^{\tau_0}_\mathcal{C}(\omega_{\operatorname{sing}}^2)|_{g_{\mathcal
       C}}\le \int_{\tau_0}^\tau \left|\mathcal{L}_{\frac{X}{2}}  \phi^{\rho}_\mathcal{C}(\omega_{\operatorname{sing}}^2)- \phi^{\rho}_\mathcal{C}(\omega_{\operatorname{sing}}^2)\right|_{g_{\mathcal{C}}}d\rho\le C|\tau-\tau_0|,
     \end{equation*}
     holds for some uniform constant $C>0$. Then the continuity on $\tau$ follows as well.
\end{proof}
\subsection{Topological semi-conjugacy}
 Thanks to the results in Corollary \ref{singular data}, there is an $\sigma_0>0$ such that for any $\sigma\le\sigma_0$ and for all $\omega_{\operatorname{sing}}\in \mathcal{K}_\sigma^b$, there is a smooth immortal K\"ahler-Ricci flow $(\omega_{\operatorname{smooth}}(t))_{t>0}$ emerging from $\omega_{\operatorname{sing}}$. 
By our estimates in Theorem \ref{estimation of KRF}, we conclude the following lemma.
\begin{lemma}\label{lemma: well-poseness of singular data}
There exist a $\sigma_0>0$ and constants $C>1,B=\{B_k>0\}_{k\in\N_0}$ that depend only on $\sigma_0$ and $b$ such that for all $\omega_{\operatorname{sing}}\in\mathcal{K}^b_\sigma$ with $\sigma\le \sigma_0$, there is a smooth immortal K\"ahler-Ricci flow $(\omega_{\operatorname{smooth}}(t))_{t>0}$ with the following properties:
    \begin{enumerate}
        \item The K\"ahler-Ricci flow $(\omega_{\operatorname{smooth}}(t))_{t> 0}$ is a $(C,B)-$bounded solution to K\"ahler-Ricci flow, i.e $(\omega_{\operatorname{smooth}}(t))_{t> 0}\in\mathcal{N}_C^B$.
        \item As $t$ tends to 0, $\omega_{\operatorname{smooth}}(t)$ converges locally and smoothly to $\omega_{\operatorname{sing}}$.
    \end{enumerate}
\end{lemma}
Indeed, we can also prove that $(\omega_{\operatorname{smooth}}(t))_{t> 0}$ is the unique bounded solution to K\"ahler-Ricci flow emerging from $\omega_{\operatorname{sing}}$ in Lemma \ref{lemma: well-poseness of singular data}.
\begin{prop}\label{prop: uniqueness of bounded sols}
    Let $(\omega_{\operatorname{bd},i}(t))_{t>0}, i=1,2$ be two bounded solutions to K\"ahler-Ricci flow such that $\omega_{\textnormal{bd},1}(t)-\omega_{\textnormal{bd},2}(t)$ locally converges to 0 as $t$ tends to $0$, then we have $\omega_{\operatorname{bd},1}(t)=\omega_{\operatorname{bd},2}(t)$ for all $t>0$.
\end{prop}
\begin{proof}
If $\dim_\C M=1$, then the uniqueness of 2D-Ricci flow in \cite{PeacheyTopping} ensures directly that $\omega_{\operatorname{bd},1}(t)=\omega_{\operatorname{bd},2}(t).$
  
    Now we assume that $\dim_\C M\ge 2$. Let $\varphi_i,i=1,2$ be the K\"ahler potentials such that
    \begin{equation*}
        \omega_{\operatorname{bd},i}(t)=\omega(t)+i\partial\bar\partial\varphi_i(t).
    \end{equation*}
    Let $C>1$ be the constant such that for $i=1,2,t>0$, we both have
    \begin{equation*}
        C^{-1}g(t)\le  g_{\operatorname{bd},i}(t)\le C g(t).
    \end{equation*}
    Define $f_{\varphi,i}(t):=t\Phi^*_tf+\frac{X}{2}\cdot\varphi_i(t)$.
    Similar computations as in \cite[Proposition 4.4]{longteng2} imply that  
    \begin{equation*}
       C^{-1}t\Phi^*_tf\le f_{\varphi,i}(t)\le C t\Phi^*_tf,
    \end{equation*}
    holds for all $t>0$. For $i=1,2$, by the K\"ahler-Ricci flow equation, we have
    \begin{equation*}
       i\partial\bar\partial\left( \frac{\partial}{\partial t}\varphi_i(t)-\log\frac{\omega_{\operatorname{bd},i}(t)^n}{\omega(t)^n}\right)=0.
    \end{equation*}
      The Liouville lemma in \cite{longteng1} together with the existence of Killing vector field $JX$  show that
    \begin{equation*}
        \frac{\partial}{\partial t}\varphi_i(t)-\log\frac{\omega_{\operatorname{bd},i}(t)^n}{\omega(t)^n}=c_i(t).
    \end{equation*}
    Here the RHS $(c_i(t))_{t>0}$ is a function that only depends on $t$. Take $C_i(t)$ as the primitive function of $c_i(t)$ and define $\varphi_i(t)-C_i(t)$ as the new $\varphi_i(t)$ then we get
    \begin{equation*}
        \frac{\partial}{\partial t}\varphi_i(t)-\log\frac{\omega_{\operatorname{bd},i}(t)^n}{\omega(t)^n}=0.
    \end{equation*}
  Let $\varphi:=\varphi_1-\varphi_2$, then we have
    \begin{equation*}
        \frac{\partial}{\partial t}\varphi(t)-\log\frac{\omega_{\operatorname{bd},1}(t)^n}{\omega_{\operatorname{bd},2}(t)^n}=0.
    \end{equation*}
    The above equation implies that when $t$ tends to 0, the function $\varphi(t)$ converge to a smooth function $\varphi(0)$ on $\mathcal{C}\setminus \{o\}$. In particular, since $(\omega_{\operatorname{bd},i}(t))_{t>0}$ has the same initial condition, then $i\partial\bar\partial \varphi(0)=0$. Use the Liouville lemma in \cite{longteng1}, since $\dim_\C M\ge2$, we conclude that $\varphi(0)$ is constant $c$ on $\mathcal{C}\setminus\{o\}$. Define $\varphi-c$ as the new $\varphi$ then we get
    \begin{equation*}
         \frac{\partial}{\partial t}\varphi(t)-\log\frac{\omega_{\operatorname{bd},1}(t)^n}{\omega_{\operatorname{bd},2}(t)^n}=0,
    \end{equation*}
    and $\varphi(t)$ converges to $0$ locally smoothly on $\mathcal{C}\setminus\{o\}$ as $t$ tends to 0.

    Observe that $(\omega_{\operatorname{bd},i})_{t>0}$ is a bounded solution, therefore, there exists a uniform constant $C>0$ such that for all $t>0$
    \begin{equation*}
        \left|\frac{\partial}{\partial t}\varphi(t)\right|\le C.
    \end{equation*}
    Since $\varphi(0)=0$ on $\mathcal{C}\setminus\{o\}$, by integration and the density of $M\setminus E$ on $M$, we have
    \begin{equation*}
        |\varphi(t)|\le Ct,
    \end{equation*}
    holds for all $t>0$.

    For all $0<s<T<\infty$, we consider $\Omega_{s,T}:=\{(x,t)\ |\ x\in M, t\in [s,T]\}$. First, we show
    \begin{equation*}
       \left( \frac{\partial}{\partial t}-\Delta_{\omega_{\operatorname{bd},i}}\right)f_{\varphi,i}=0.
    \end{equation*}
    Indeed, we compute
    \begin{equation*}
        \begin{split}
             \left( \frac{\partial}{\partial t}-\Delta_{\omega_{\operatorname{bd},i}}\right)f_{\varphi,i}&=\Phi_t^*f-\frac{X}{2}\cdot \Phi_t^*f+\frac{X}{2}\cdot \frac{\partial}{\partial t}\varphi_i-\Delta_{\omega_{\operatorname{bd},i}}f_{\varphi,i}\\
             &=\Phi_t^*f-\frac{X}{2}\cdot \Phi_t^*f+\Delta_{\omega_{\operatorname{bd},i}}f_{\varphi,i}-\Delta_{\omega(t)}t\Phi_t^*f-\Delta_{\omega_{\operatorname{bd},i}}f_{\varphi,i}\\
             &=\Phi_t^*(f-|\partial f|_g^2-\Delta_\omega f)\\
             &=0.
        \end{split}
    \end{equation*}
    Observe that, by the elementary inequality, we have
    \begin{equation*}
       \Delta_{\omega_{\operatorname{bd},1}(t)}\varphi(t)\le  \frac{\partial}{\partial t}\varphi(t)\le \Delta_{\omega_{\operatorname{bd},2}(t)}\varphi(t).
    \end{equation*}
    For all $\delta>0$, consider the function $\varphi-\delta f_{\varphi,2}-\delta t$ on $\Omega_{s,T}$. On the one hand, we have
    \begin{equation*}
         \left( \frac{\partial}{\partial t}-\Delta_{\omega_{\operatorname{bd},2}}\right)(\varphi-\delta f_{\varphi,2}-\delta t)\le -\delta<0.
    \end{equation*}
    On the other hand, the function $\varphi-\delta f_{\varphi,2}-\delta t$ tends to $-\infty$ uniformly at spatial infinity since $|\varphi|\le Ct\le CT$. Then by the maximum principle, for all $\delta>0$, we have
    \begin{equation*}
        \max_{\Omega_{s,t}}(\varphi-\delta f_{\varphi,2}-\delta t)\le \max_M(\varphi(s)-\delta f_{\varphi,2}(s)-\delta s)\le \sup_{M}\varphi(s)\le Cs.
    \end{equation*}
    Let $\delta$ tend to 0, we get
    \begin{equation*}
        \sup_{\Omega_{s,T}}\varphi\le Cs.
    \end{equation*}
    Respectively, by applying the above argument to $\varphi+\delta f_{\varphi,1}+\delta t$, we get
    \begin{equation*}
        \inf_{\Omega_{s,T}}\varphi\ge -Cs.
    \end{equation*}
    Let $T$ tend to $\infty$ and let $s$ tend to 0, then it follows that $\varphi(t)\equiv0$ for all $t>0$. Thus, we have
    \begin{equation*}
        \omega_{\operatorname{bd},1}(t)-\omega_{\operatorname{bd},2}(t)=0,
    \end{equation*}
    holds for all $t>0$ as expected.
    
\end{proof}

\begin{corollary}\label{coro: definition of F}
   Let $\sigma_0>0$ be as in Lemma \ref{lemma: well-poseness of singular data}, then for any $b=\{b_k>0\}_{k\in \N^*}$, $\sigma\le\sigma_0$, there exist constants $C>1, B=\{B_k>0\}_{k\in\N^*}$ and a natural injection:
   \begin{equation*}
       F:\mathcal{K}_\sigma^b\mapsto\mathcal{N}_C^B,
   \end{equation*}
   such that for all $\omega_{\operatorname{sing}}$, $F(\omega_{\operatorname{sing}})=(\omega_{\operatorname{smooth}}(t))_{t>0}$ is the unique bounded solution to K\"ahler-Ricci flow with initial data $\omega_{\operatorname{sing}}$.

\end{corollary}
\begin{proof}
    Since $\sigma\le\sigma_0$, then thanks to Corollary \ref{singular data}, for each $\omega_{\operatorname{sing}}\in \mathcal{K}_{\sigma}^b$, there is a bounded solution to K\"ahler-Ricci flow $(\omega_{\operatorname{smooth}}(t))_{t>0}$ with initial data $\omega_{\operatorname{sing}}$. Due to Proposition \ref{prop: uniqueness of bounded sols}, this flow is unique.

\end{proof}

\begin{theorem}[Topological semi-conjugacy]
    The map $F: \mathcal{K}_{\sigma}^b\mapsto\mathcal{N}^B_{C}$ in Lemma \ref{coro: definition of F} is continuous. Moreover, $F$ is a topological semi-conjugacy between $(\mathcal{K}^b_{\sigma},d_{\mathcal{C}},(\phi^\tau_{\mathcal{C}})_{\tau\in\R})$ and $(\mathcal{N}^b_{\sigma},d,(\phi^\tau)_{\tau\in\R})$, i.e, $\forall \tau\in\R$ the following diagram commute:
\begin{equation*}
    \begin{tikzcd}
\mathcal{K}_{\sigma}^b \arrow{r}{\phi_{\mathcal{C}}^\tau} \arrow{d}{F}
& \mathcal{K}_{\sigma}^b \arrow{d}{F} \\
\mathcal{N}_C^B \arrow{r}{\phi^\tau}
& \mathcal{N}_C^B
\end{tikzcd}
\end{equation*}
\end{theorem}
\begin{proof}
    First, we verify that
    \begin{equation*}
        \phi^\tau\circ F=F\circ \phi_\mathcal{C}^\tau,\quad \forall \tau\in\R.
    \end{equation*}
    Let $\omega_{\operatorname{sing}}\in\mathcal{K}_{\sigma}^b$, let $\omega_{\operatorname{smooth}}(t)=F(\omega_{\operatorname{sing}})\in\mathcal{N}_C^B $, then it follows that
    \begin{equation*}
         \phi^\tau\circ F(\omega_{\operatorname{sing}})(t)=e^{-\tau}\Phi_{e^{-\tau}}^*\omega_{\operatorname{smooth}}(e^\tau t).
    \end{equation*}
    Observe that $(e^{-\tau}\Phi_{e^{-\tau}}^*\omega_{\operatorname{smooth}}(e^\tau t))_{t>0}$ is a bounded solution to K\"ahler-Ricci flow with initial data $e^{-\tau}\Phi_{e^{-\tau}}^*\omega_{\operatorname{sing}}=\phi^\tau_{\mathcal{C}}(\omega_{\operatorname{sing}})$, then we have
    \begin{equation*}
        F\circ \phi^\tau_{\mathcal{C}}(\omega_{\operatorname{sing}})(t)=e^{-\tau}\Phi_{e^{-\tau}}^*\omega_{\operatorname{smooth}}(e^\tau t)= \phi^\tau\circ F(\omega_{\operatorname{sing}})(t).
    \end{equation*}
    It suffices to show that $F$ is actually continuous.
   Let's fix $\omega_{\operatorname{sing}}\in \mathcal{K}^b_\sigma$, $F(\omega_{\operatorname{sing}})=\omega_{1}\in \mathcal{N}_C^B$, and an open neighborhood $U^\varepsilon(\omega_1,\mathscr{R}_{[t_0,t_1]}^L)$ of $\omega_1$. Let $\omega_{\operatorname{sing}}'\in\mathcal{K}^b_\sigma$ and $|\omega_{\operatorname{sing}}'-\omega_{\operatorname{sing}}|_{\omega_\mathcal{C}}\le \delta$ on $\mathscr{A}_{r_0,R_0}$ with $\delta>0, 0<r_0<R_0<\infty$ to be determined. Let $\varphi_{0}$ be a smooth function on $\mathcal{C}\setminus\{o\}$ such that
    \begin{equation}
        \omega_{\operatorname{sing}}-\omega_{\operatorname{sing}}'=i\partial\bar\partial \varphi_0.
    \end{equation}
    By the definition of $\mathcal{K}^b_\sigma$, there is a uniform constant $\kappa>0$ such that  
    \begin{equation*}
        |\varphi_0|\le\kappa r^2,\quad |\nabla^{g_{\mathcal{C}}}\varphi_0|^2_{g_{\mathcal{C}}}\le\kappa r^2.
    \end{equation*}
    On $\mathscr{A}_{r_0,R_0}$, we compute, since $\varphi_0$ is $JX-$invariant,
    \begin{equation*}
        \begin{split}
           \left| \frac{X}{2}\cdot(X\cdot \varphi_0)\right|=|i\partial\bar\partial \varphi_0(X,JX)|\le \delta r^2.
        \end{split}
    \end{equation*}
    By integration, for $r\le R_0$ we get
    \begin{equation*}
       |( X\cdot\varphi_0)(r)-(X\cdot\varphi)(r_0)|\le \delta (r^2-r_0^2)\le \delta r^2.
    \end{equation*}
    Hence,
    \begin{equation*}
        |X\cdot\varphi_0|\le \delta r^2+\kappa r_0^2\le \delta^2 r^2+\kappa r r_0.
    \end{equation*}
    Still, by integration, for $r\le R_0$ we get
    \begin{equation*}
        |\varphi_0-\varphi_0(r_0)|\le \delta\frac{r^2}{2}+\kappa r_0(r-r_0)\le \delta\frac{r^2}{2}+\kappa r_0r\le \delta r^2+2\kappa^2 \delta^{-1}r_0^2.
    \end{equation*}
    Hence,
    \begin{equation*}
        |\varphi_0|\le \delta r^2+(\kappa+2\kappa^2 \delta^{-1})r_0^2.
    \end{equation*}
We observe that the above bound also works on $\{r^2\le r_0^2\}$.
    
    Let $F(\omega_{\operatorname{sing}}')=(\omega_1'(t))_{t>0}$, let $\varphi(t)\in C^\infty(M\times (0,\infty))$ a smooth $JX-$invariant function that is defined as in Proposition \ref{prop: uniqueness of bounded sols} such that 
    \begin{equation*}
        \begin{split}
           &\frac{\partial}{\partial t}\varphi(t)=\log\frac{\omega_{1}'(t)^n}{\omega_1(t)^n},\\
           &\omega_1(t)=\omega_1(t)+i\partial\bar\partial \varphi(t).
        \end{split}
    \end{equation*}
    Since both $(\omega_1'(t)))_{t>0}$ and $(\omega_1(t))_{t>0}$ are bounded solutions, there is a uniform constant $C>0$ such that
    \begin{equation*}
        \left|\frac{\partial}{\partial t}\varphi(t)\right|\le C.
    \end{equation*}
    By integration, we claim that $\varphi(t)$ converges to a smooth function on $\mathcal{C}\setminus\{o\}$ and we have
    \begin{equation*}
        |\varphi(t)-\varphi(0)|\le Ct,
    \end{equation*}
    holds on $\mathcal{C}\setminus\{o\}$. In addition, by the initial condition, we also have
    \begin{equation*}
        i\partial\partial(\varphi(0)-\varphi_0)=0.
    \end{equation*}
    Since $\varphi(0)-\varphi_0$ is locally bounded around $o$, then by Liouville lemma in \cite{longteng1}, we conclude that $\varphi(0)-\varphi_0$ is constant. Without loss of generality, we could suppose this constant to be 0.

    Both of the flows have quadratic curvature decay, then on $\{r^2=R_0^2\}\times [0,t_1]$ with $t_1>0$ defined in $\mathscr{R}_{[t_0,t_1]}^L$, we have
    \begin{equation*}
        \left|\frac{\partial}{\partial t}\varphi(t)\right|\le \left|\log\frac{\omega_{\operatorname{sing}}'^n}{\omega_{\operatorname{sing}}^n}\right|+C\frac{t}{R_0^2}\le C(\delta +tR_0^{-2}),
    \end{equation*}
    holds for some uniform constant $C>0$. by integration, we have
    \begin{equation*}
        |\varphi(t)-\varphi_0|\le C(\delta t_1 +t_1^2R_0^{-2}).
    \end{equation*}
    Therefore, on $\{r^2=R_0^2\}\times [0,t_1]$ we get
    \begin{equation*}
        |\varphi(t)|\le \delta r^2+C\delta t_1 +Ct_1^2R_0^{-2}+(\kappa+2\kappa^2 \delta^{-1})r_0^2.
    \end{equation*}
    Let $f_{\varphi}$ and $f'_{\varphi}$ be the Hamiltonian of $X$ with respect to $\omega_1$ and $\omega_1'$. Then by previous estimates, there exists a uniform constant $D>1$ such that
    \begin{equation*}
        D^{-1}t\Phi_t^*f\le f_{\varphi}\le D t\Phi_t^*f;\quad D^{-1}t\Phi_t^*f\le f_{\varphi}'\le D t\Phi_t^*f.
    \end{equation*}
    Therefore, we conclude that on $\{r^2= R_0^2\}\times [0,t_1]\cup\{r^2\le R_0^2\}\times \{t=0\}$, we have
    \begin{equation*}
        \begin{split}
            &\varphi(t)\le D\delta f_\varphi(t)+C\delta t_1 +Ct_1^2R_0^{-2}+(\kappa+2\kappa^2 \delta^{-1})r_0^2;\\
            &\varphi(t)\ge -D\delta f_\varphi'(t)-C\delta t_1 -Ct_1^2R_0^{-2}-(\kappa+2\kappa^2 \delta^{-1})r_0^2.
        \end{split}
    \end{equation*}
    Then the maximum principle implies 
    \begin{equation*}
        \begin{split}
            &\varphi(t)-D\delta f_\varphi(t)\le C\delta t_1 +Ct_1^2R_0^{-2}+(\kappa+2\kappa^2 \delta^{-1})r_0^2; \\
            &\varphi(t)-D\delta f_\varphi(t)'\le -C\delta t_1 -Ct_1^2R_0^{-2}-(\kappa+2\kappa^2 \delta^{-1})r_0^2.
        \end{split}
    \end{equation*}
    we conclude that on $\{r^2\le R_0^2\}\times [0,t_1]$, we have
    \begin{equation*}
        |\varphi(t)|\le D^2 \delta t\Phi^*_tf+C\delta t_1 +Ct_1^2R_0^{-2}+(\kappa+2\kappa^2 \delta^{-1})r_0^2.
    \end{equation*}
    Now we take $r_0=\delta <1$ and $R_0^2=t_1^2\delta^{-1}$, then there exists a constant $C>0$ such that on $\{r^2\le R_0^2\}\times [0,t_1]$, we have
    \begin{equation*}
        |\varphi(t)|\le D^2 \delta t\Phi^*_tf+C\delta.
    \end{equation*}
    Then on $\{r^2\le R_0^2\}\times [t_0,t_1]$, we have
    \begin{equation*}
          |\varphi(t)|\le D^2 \delta t\Phi^*_tf+C\delta t t_0^{-1}.
    \end{equation*}
    By applying interpolation inequality in \cite{ChenHallgrenLucas}, fro $\delta$ sufficiently small, we get
    \begin{equation*}
        |\omega_1(t)-\omega_1(t)|_{\omega(t)}\le C\delta^{1/3},
    \end{equation*}
    holds on $\{r^2\le R_0^2/4\}\times [t_0,t_1]$. Here the constant $C$ depends on $t_0,t_1$. Taking $\delta<<1$ such that $C\delta^{1/3}\le \varepsilon$ and $R_0^2\ge L$ completes the proof.
\end{proof}
\subsection{Periodic points and Ricci breathers}
Let $\omega_1\in(\mathcal{N}_C^B,d,(\phi^\tau)_{\tau\in\R})$,
notice that
$\phi^\tau(\omega_1)=\omega_1$ for all $\tau\ge 0$ if and only if $(M,\omega_(1),X)$ is an expanding K\"ahler-Ricci soliton. It turns out that the fixed points of $(\phi^\tau)_{\tau\ge 0}$ are self-similar solutions to K\"ahler-Ricci flow, but the set of periodic points contains more elements.
\begin{theorem}\label{theorem non breathers}
   For each $\lambda> 1$, on the AC gradient K\"ahler-Ricci expander $(M,g,X)$ there are non-trivial Ricci breathers, which are immortal K\"ahler-Ricci flow with type III singularity, such that each $\tilde \omega(t)_{t>0}$ of them satisfies
   \begin{equation*}
       \lambda\Phi_\lambda^*\tilde \omega\left(\frac{t}{\lambda}\right)=\tilde\omega(t).
   \end{equation*}
Here $\Phi_t$ is the flow of $-\frac{r\partial_r}{2t}:=-\frac{X}{2t}$ for $t>0$.
   
   In particular, $(\tilde \omega(t))_{t>0}$ is a periodic point of the dynamical system $(\mathcal{N}_C^B,d,(\phi^\tau)_{\tau\in\R})$.
\end{theorem}
\begin{proof}
Suppose that $\lambda=e$, indeed, the following method works for any $\lambda>1$. For all smooth non-constant periodic function $\psi\in C^\infty(\R)$ with period $1$ on $\R$ such that \begin{equation*}
    \sup_{\R}|\frac{d^k}{dx^k}\psi|\le C_k,
\end{equation*}
consider \begin{equation*}
    \omega_{\operatorname{sing}}=\omega_\mathcal{C}+\varepsilon i\partial\bar\partial(r^2\psi(\log(r^2))).
\end{equation*}
We observe that \begin{equation*}
    e\Phi_{e}^* \omega_{\operatorname{sing}}= \omega_{\operatorname{sing}}.
\end{equation*}
This 2-form satisfies the conditions in Corollary \ref{singular data} if $\varepsilon$ is sufficiently small. Thanks to Corollary \ref{singular data} and Proposition \ref{prop: uniqueness of bounded sols}, there is a unique bounded solution to K\"ahler-Ricci flow $(\omega_{\operatorname{smooth}}(t))_{t>0}$ with initial data $\omega_{\operatorname{sing}}$. Since $ (e\Phi_e^*\omega_{\operatorname{smooth}}(t/e))_{t>0}$ is also a bounded solution to K\"ahler-Ricci flow $(\omega_{\operatorname{smooth}}(t))_{t>0}$ with initial data $\omega_{\operatorname{sing}}$, then we must have:
\begin{equation*}
    e\Phi_e^*\omega_{\operatorname{smooth}}(t/e)=\omega_{\operatorname{smooth}}(t).
\end{equation*}
 
Now we prove $\omega_{\operatorname{smooth}}(t)_{t>0}$ is not a self-similar solution. Since this flow is only defined for positive time, if it is a self-similar solution, it must be determined by a Ricci expander. 
Suppose that 
\begin{equation*}
   g_{\operatorname{smooth}}(t)=g_E(t),
\end{equation*}
with $(M,g_E,Y)$ being a K\"ahler-Ricci expander (we are not assuming it is gradient) with $Y$ being a real-holomorphic vector field. Recall the soliton equation:
\begin{equation*}
    g_E(t)+t\Ric(g_E(t))=\frac{1}{2}\mathcal{L}_Yg_E(t),
\end{equation*}
and let $t$ tend to 0, we get
\begin{equation}\label{contradiction eq}
    g_{\operatorname{sing}}=\frac{1}{2}\mathcal{L}_Yg_{\operatorname{sing}}.
\end{equation}
Since $g_E$ is $JX-$invariant, it follows that we can always assume that $Y$ is $JX-$invariant by considering its average under $JX-$action. Therefore, we have $[Y,JX]=0$. Since $Y$ is a real-holomorphic vector field, we also have $J[Y,X]=[Y,JX]=0$. In this case,
\begin{equation*}
    \begin{split}
        0=\frac{1}{2}\mathcal{L}_{Y}g_{\operatorname{sing}}-g_{\operatorname{sing}}&=i\partial\bar\partial (\frac{Y}{2}\cdot (r^2/2+\varepsilon r^2\psi(\log r^2))-r^2/2-\varepsilon r^2\psi(\log r^2)):=i\partial\bar\partial\kappa.
    \end{split}
\end{equation*}
Since $\kappa$ is $JX-$invariant and locally bounded, we claim that $\kappa\equiv 0$. Locally, on an open neighborhood $U=(r_0-\delta,r_0+\delta)\times U_S$ with $U_S$ being an open set of link $S$, the vector field $Y$ can be expressed in the following way:
\begin{equation*}
    Y=aX+\sum_{i=1}^{2n-1}a_iV_i,
\end{equation*}
with $a, a_i$ being smooth functions on $U$ and $\{V_i\}_{i=1}^{2n-1}$ being the basis of $TU_S$. Since $[Y,X]=0$, locally we compute
\begin{equation*}
    \begin{split}
        0=[Y,X]=-X(a)X-\sum_{i=1}^{2n-1}X(a_i)V_i.
    \end{split}
\end{equation*}
We conclude that $X(a)=0$ and $X(a_i)=0$ for all $i\le2n-1$. On the one hand, locally we compute
\begin{equation*}
    \begin{split}
        &\frac{Y}{2}\cdot (r^2/2+\varepsilon r^2\psi(\log r^2))-r^2/2-\varepsilon r^2\psi(\log r^2)\\
        =&a(r^2/2+\varepsilon r^2\psi(\log r^2)+\varepsilon r^2\psi'(\log r^2))-r^2/2-\varepsilon r^2\psi(\log r^2).
    \end{split}
\end{equation*}
On the other hand, it is 0, we then conclude that
\begin{equation*}
    a(1+2\varepsilon\psi(\log r^2)+2\varepsilon\psi'(\log r^2))=1+2\varepsilon\psi(\log r^2).
\end{equation*}
We could take $\varepsilon>0$ sufficiently small such that $1+2\varepsilon\psi(\log r^2)+2\varepsilon\psi'(\log r^2)>0$, then we have
\begin{equation*}
    a=\frac{1+2\varepsilon\psi(\log r^2)}{1+2\varepsilon\psi(\log r^2)+2\varepsilon\psi'(\log r^2)}.
\end{equation*}
Since $X(a)=0$, in particular, $\frac{\partial}{\partial r}a=0$, therefore, a necessary condition of $(g_{\operatorname{smooth}}(t))_{t>0}$ being self-similar is
\begin{equation*}
    \frac{d}{dx}\left(\frac{1+2\varepsilon\psi(x)}{1+2\varepsilon\psi(x)+2\varepsilon\psi'(x)}\right)=0.
\end{equation*}
Define $u=1+2\varepsilon \psi(x)$, we then have
\begin{equation*}
   u=c(u+u'),
\end{equation*}
holds for some constant $c\neq0$. By integration, we then have
\begin{equation*}
    u(x)=e^{(c^{-1}-1)x}u(0),
\end{equation*}
which is not periodic unless $c=1$. When $c=1$, then $u=1+2\varepsilon\psi(x)\equiv\textnormal{constant}$, this still leads to contradiction.

\end{proof}
\begin{remark}
    See \cite{ToppingNonBreather} for cylindrical non-trivial Ricci breather. 

    In the above Theorem, we have proved the above solution to K\"ahler-Ricci flow is not a self-similar solution associated with a K\"ahler-Ricci expander. In fact, if the AC gradient K\"ahler-Ricci expander is one of FIK expanders, Cao's expanders and Gaussian expander, then the Ricci breathers that we constructed are not self-similar solutions associated with Ricci expanders.

    To see this, suppose that 
\begin{equation*}
   g_{\operatorname{smooth}}(t)=g_E(t),
\end{equation*}
with $(M,g_E,Y)$ being a Ricci expander (we are not assuming it is gradient or K\"ahler). Recall the soliton equation:
\begin{equation*}
    g_E(t)+t\Ric(g_E(t))=\frac{1}{2}\mathcal{L}_Yg_E(t),
\end{equation*}
and let $t$ tend to 0, we get
\begin{equation}\label{contradiction eq 2}
    g_{\operatorname{sing}}=\frac{1}{2}\mathcal{L}_Yg_{\operatorname{sing}}.
\end{equation}
Since $g_{\operatorname{sing}}$ is $U(n)-$invariant, so is $g_{\operatorname{smooth}}(t)$ by its construction. Therefore, we can always assume that $Y$ is $U(n)-$invariant. In this case, there exists a smooth function $a$ that only depends on $r$ such that $Y=a(r)X$.
Let $X^\flat$ be the flat of $X$ with respect to metric $g_{\operatorname{sing}}$. Since $g_{\operatorname{sing}}$ is $JX-$invariant, it follows that
\begin{equation*}
    X^\flat=d(\frac{r^2}{2}+\frac{X}{2}\cdot (\varepsilon r^2\psi(\log r^2)):=d(\psi_{\operatorname{sing}}(r)).
\end{equation*}
The metric $g_{\operatorname{sing}}$ is K\"ahler, the equation \eqref{contradiction eq 2} implies that
\begin{equation*}
    0= (\mathcal{L}_{Y}g_{\operatorname{sing}})_{ij}=\partial_i a\partial_j\psi_{\operatorname{sing}}+\partial_j a\partial_i\psi_{\operatorname{sing}}=2\partial_ir\partial_j r  \frac{d a}{d r}\frac{d\psi_{\operatorname{sing}}}{d r}.
\end{equation*}
Therefore,
\begin{equation*}
    \frac{d a}{d r}\frac{d\psi_{\operatorname{sing}}}{d r}=0,
\end{equation*}
which implies that $a\equiv constant$. In this case, we have
\begin{equation*}
    \frac{1}{2}\mathcal{L}_{Y}g_{\operatorname{sing}}=\frac{1}{2}a\mathcal{L}_{X}g_{\operatorname{sing}}=g_{\operatorname{sing}},
\end{equation*}
which 
leads to the same contradiction as in the proof of Theorem \ref{theorem non breathers}.
\end{remark}

\bibliographystyle{alpha}
\bibliography{references}
\end{document}